\documentclass{article}
\usepackage{amsfonts}
\usepackage{graphicx}
\usepackage{amsmath}
\usepackage{amssymb}
\usepackage{float,color,ulem}
\usepackage{epsf,epsfig,subfigure}
\usepackage{chngcntr}
\usepackage{apptools}

\newtheorem{theorem}{Theorem}

\newtheorem{claim}[theorem]{Claim}

\newtheorem{assumption}[theorem]{Assumption}

\newtheorem{corollary}[theorem]{Corollary}

\newtheorem{lemma}[theorem]{Lemma}

\newtheorem{proposition}[theorem]{Proposition}
\newtheorem{remark}[theorem]{Remark}

\newenvironment{proof}[1][Proof]{\noindent\textbf{#1.} }{\ \rule{0.5em}{0.5em}}
\newcommand{\R}{\mathbb{R}}
\newcommand{\EE}{\mathbb{E}}
\newcommand{\M}{\mathcal{M}}
\renewcommand{\P}{\mathcal{P}}
\newcommand{\C}{\mathcal{C}}
\renewcommand{\d}{{\rm d}}
\AtAppendix{\counterwithin{theorem}{section}}

\begin{document}

\title{Self-similar behaviour of a non-local diffusion equation with time delay}
\author{\bigskip Arnaud Ducrot$^{a}$ and Alexandre Genadot$^b$\\
$^a${\small email: arnaud.ducrot@univ-lehavre.fr}\\
{\small Normandie Univ, UNIHAVRE, LMAH, FR-CNRS-3335, ISCN, 76600 Le Havre, France}\\
$^b${\small email: alexandre.genadot@math.u-bordeaux.fr}\\
{\small Univ. Bordeaux, IMB, UMR 5251, F-33076 Bordeaux, France}\\
{\small CNRS, IMB, UMR 5251, F-33400 Talence, France}\\
{\small INRIA Bordeaux-Sud Ouest, Team CQFD, F-33400 Talence, France}}
\maketitle

\begin{abstract}
We study the asymptotic behaviour of solutions of a class of linear non-local measure-valued differential equations with time delay. Our main result states that the solutions asymptotically exhibit a parabolic like behaviour in the large times, that is precisely expressed in term of heat kernel. Our proof relies on the study of a -- self-similar -- rescaled family of solutions. We first identify the asymptotic behaviour of the solutions by deriving a convergence result in the sense of the Young measures. Then we strengthen this convergence by deriving suitable fractional Sobolev compactness estimates.
As a by-product, our main result allows to obtain asymptotic results for a class of piecewise constant stochastic processes with memory.
\end{abstract}


\section{Introduction}

Mathematical models with non-local diffusion arise in various applicative fields including physics and biology. Here, by non-local diffusion we mean convolution equation of the form
$$
\frac{\partial u(t,x)}{\partial t}=\int_{\R^N} J(x-y)\left[u(t,y)-u(t,x)\right]\d y,\;t>0\text{ and }x\in\R^N,
$$
wherein $J$ denotes a probability kernel on $\R^N$, for some given and fixed integer $N\geq 1$.
As described by Fife in \cite{Fife1}, if $u=u(t, x)$ denotes the density of a population at time $t$ and spatial location $x\in\R^N$, the first term in the right hand side of this equation describes the rate at which individuals jump to $x$ from all other locations $y$ weighted with the probability kernel $J(x-y)$, while the second term $u(t,x)=u(t,x)\int_{\R^N}J(\d y)$, corresponds to the rate at which individuals are leaving the location $x$ to move to some other places. 

As mentioned above, such a motion equation arises in various applications. We refer to \cite{Bates, 6, 7, 8, Fife1, Fife2, Jin-Zhao, Lutscher} and the references therein for analysis of models coming from physics, mathematical biology and population dynamics. We would like to emphasize that in such models, the jumps of particles or individuals are supposed to be instantaneous. 

We can extend the above non-local diffusion equation by taking into account the travel time to jump from $y$ to $x$. Typically, the travel time depends on the distance to travel ($x-y$) so that the above non-local motion law can be extended as follows. For $t>0$ and $x\in\R^N$, a typical non-local diffusion equation with time delay reads as
\begin{equation}\label{typical}
\frac{\partial u(t,x)}{\partial t}=\int_{\R^N} J(x-y)\left[u(t-\tau(x-y),y)-u(t,x)\right]\d y,
\end{equation}
wherein $\tau(x-y)$ denotes the time needed to jump from $y$ to $x$. This work is concerned with the asymptotic analysis of such a non-local diffusion equation with time delay. To that aim, we will consider more general version of this problem such as the equation
\begin{equation}\label{eq-delay}
\frac{\partial u (t,x)}{\partial t}=\int_{[-1,0]\times \R^N}\left[u(t+\theta,x-y)-u(t,x)\right]Q(\d \theta,\d y),\;t>0,
\end{equation}
supplemented with a suitable initial data $u(\theta,x)$ for $\theta\in [-1,0]$ and $x\in\R^N$.
In the above equation, $Q=Q(\d\theta,\d y)$ denotes a probability distribution on the infinite strip $S:=[-1,0]\times \R^N$. 
This distribution captures the information about both the jump process and the time needed to perform jumps.
For instance, if we choose $Q(\d\theta,\d y)=\delta_{-\tau(y)}(\d\theta)\otimes J(y)\d y$, we recover the typical equation presented in \eqref{typical} above.
Below we will describe a first stochastic representation of such a problem.

Observe that when $Q(\d\theta,\d y)=\delta_0(\d \theta)\otimes J(\d y)$, Equation \eqref{eq-delay} becomes
\begin{equation*}
\frac{\partial u(t,x)}{\partial t}=\int_{\R^N}\left[u(t,x-y)-u(t,x)\right]J(\d y).
\end{equation*}
This equation corresponds to the usual random walk equation with instantaneous jumps and associated with the jump measure $J$.
Under suitable assumptions on this jump measure $J$, namely the existence of second moments and the non degeneracy of the covariance matrix, the positive solutions of this equation satisfy a central limit theorem. We refer to Ignat-Rossi \cite{Ignat-Rossi} for a detailed study of this equation using Fourier analysis. We also refer to Chasseigne et al in \cite{5} for a refined study in the case where the Fourier transform of the kernel $J$ has lower regularity close to zero and the connexion with the self-similar solutions of the heat equation with fractional Laplace operator. More general stochastic representations of the solution to equation \eqref{eq-delay} are considered in Section \ref{sec:proba}. The corresponding stochastic processes are piecewise constant processes with memory. The main result of the present paper allows to study the asymptotic behaviour of such stochastic processes.

Indeed, the aim of this work is to study the large time behaviour for the delayed equation \eqref{eq-delay}.
For the sake of generality and also for future applicative uses, we shall consider a slightly more general version of such an equation.
In order to present the equation we will consider throughout this work, we denote by $\mathcal P(\R^N)$ the set of Borel probability measure on $\R^N$. In the sequel, we will simply write $\mathcal P$ instead of $\mathcal P(\R^N)$ when there is no possible confusion.  This space is endowed with the usual metrizable narrow topology (see Appendix A). \\
Let $\alpha:[0,\infty)\times [-1,0]\to (0,\infty)$ be a given bounded and continuous function.
We consider the following problem:
\begin{equation}\label{eq-retard}
\begin{cases}
\displaystyle \dfrac{\partial}{\partial t} u(t,\d x)=\int_{S}\alpha(t,\theta)\left[u(t+\theta,\d x-y)-u(t,\d x)\right]Q(\d \theta,\d y),\;t>0,\\
\displaystyle u(\theta,\d x)=u^0(\theta,\d x),\;\forall \theta\in [-1,0]\text{ and }u^0\in \C\left([-1,0];\mathcal P(\R^N)\right). 
\end{cases}
\end{equation} 
Similarly to \eqref{eq-delay}, in the above problem $Q$ belongs to $\mathcal P(S)$, the set of Borel probability measures on the strip $S=[-1,0]\times \R^N$. 

Remark that the above equation is posed for probability measures allowing us to handle lattice equations with time delay. Indeed, consider for instance the case where
\begin{equation*}
Q(\d\theta,\d y)=\sum_{j\in \mathbb Z}\alpha_j\delta_{-\tau_j}(\d\theta)\otimes \delta_j(\d y)\text{ with }-\tau_j\in [-1,0],\;\alpha_j\geq 0\text{ and }\sum_{j\in\mathbb Z}\alpha_j=1,
\end{equation*}
then, when equipped with initial data of the form
$$
u(\theta,\d x)=\sum_{i\in\mathbb Z} u_i^0(\theta)\delta_i(\d x)\in \C\left([-1,0];\mathcal P(\R^N)\right),
$$
the solutions of the above equation takes the form $u(t,\d x)=\sum_{i\in\mathbb Z} u_i(t)\delta_i(\d x)$ where the functions $\left(u_i(t)\right)_{i\in\mathbb Z}$ satisfy the delayed lattice system of equations
\begin{equation*}
\begin{cases}
\displaystyle u_i'(t)=\sum_{j\in \mathbb Z} \alpha_{j} u_{i-j}\left(t-\tau_{j}\right)-u_i(t),\;t>0,\\
u_i(\theta)=u_i^0(\theta),\;\theta\in [-1,0]
\end{cases},\;\forall i\in\mathbb Z.
\end{equation*}

We now come back to Problem \eqref{eq-retard}. Before stating our main result, we describe the main set of assumptions that will be used to study the asymptotic behaviour of \eqref{eq-retard}.
\begin{assumption}\label{ASS1}
We assume that:
\begin{itemize}
\item[(i)] The function $\alpha=\alpha(t,\theta)$ defined from $[0,\infty)\times [-1,0]$ into $(0,\infty)$ is bounded and continuous and there exist $M>0$ and $\beta>0$ such that 
\begin{equation*}
|\alpha(t,\theta)-\alpha_\infty(\theta)|\leq Me^{-\beta t},\;\forall t\geq 0,\;\theta\in [-1,0],
\end{equation*}
wherein $\alpha_\infty=\alpha_\infty(\theta)>0$ is a continuous function on $[-1,0]$.
\item[(ii)] The function $(\theta,x)\mapsto |x|$ belongs to $L^2\left(S; Q(\d \theta,\d x)\right)$.
Here we use the symbol $|\cdot|$ to denote the Euclidean norm in $\R^N$.
\end{itemize}
\end{assumption}

Throughout this work, we fix $u^0\in \C\left([-1,0];\mathcal P(\R^N)\right)$ and we consider the solution $u\equiv u(t,\d x)$ of \eqref{eq-retard} associated with the initial data $u^0$. Here, by a solution we mean an element $\mu=\mu(t,\d x)\in \C\left([-1,\infty);\mathcal P(\R^N)\right)$ that satisfies:
\begin{itemize}
\item [(i)] for all $\theta\in [-1,0]$, one has $\mu(\theta,\d x)=u^0(\theta,\d x)$;

\item[(ii)] for any $f\in \mathcal C_b(\R^N)$, the space of bounded and continuous functions on $\R^N$, one has:
\begin{itemize}
\item [(iia)] the map $t\mapsto \int_{\R^N} f(x) \mu(t,\d x)$ is continuously differentiable for $t\in (0,\infty)$;
\item[(iib)] for all $t>0$ one has
\begin{equation*}
\begin{split}
\frac{\d }{\d t}\int_{\R^N} f(x) \mu(t,\d x)=&\int_{S}\alpha(t,\theta)\left[\int_{\R^N}\hspace{-0.1cm}f(x+z)\mu(t+\theta,\d x)\right]Q(\d \theta,\d z)\\
&-\int_{\R^N}f(x)\mu(t,\d x)\int_{S}\alpha(t,\theta)Q(\d \theta,\d z).
\end{split}
\end{equation*}
\end{itemize}
\end{itemize}
Note that the existence and uniqueness of such a solution simply follows from a usual contraction fixed point argument in the metric space $\C\left([-1,T];\mathcal P(\R^N)\right)$ for $T>0$.

Our main result reads as follows.

\begin{theorem}[Self-similar behaviour]\label{THEO2}
Let Assumption \ref{ASS1} be satisfied.
Then there exists a vector ${\bf K}\in\R^N$ and a probability measure $\pi\in\mathcal P(\R^N)$ with zero mean
value such that the function $t\mapsto \mu(t,\d x)\in \C\left([-1,\infty);\mathcal P(\R^N)\right)$ satisfies the following asymptotic behaviour:
for all test functions $f\in\mathcal C_b\left(\R^N\right)$ one has
\begin{equation*}
\lim_{t\to\infty}\int_{\R^N}f\left(\frac{x-{\bf K}t}{\sqrt t}\right)\mu(t,\d x)=\int_{\R^N}f(x)\pi(\d x).
\end{equation*}
The vector ${\bf  K}$ is given by
\begin{equation}\label{9}
{\bf  K}=\frac{1}{1+\Gamma}\int_S \alpha_\infty (\theta) z Q(\d\theta,\d z)\text{ with }\Gamma:=\int_S (-\theta)\alpha_\infty(\theta) Q(\d\theta,\d z),
\end{equation}
while $\pi(\d x)=\Pi(1,\d x)$ where $\Pi(t,\d x)\in \C(\R^+,\mathcal P)$ denotes the unique (tempered distribution) solution of the heat equation
\begin{equation}\label{10}
\partial_t \Pi=\frac12{\rm div}\,\left(\frac{D_0}{1+\Gamma}\nabla \Pi\right)+\delta_{t=0}\otimes \delta_0(\d x).
\end{equation}
Herein $D_0$ denotes the non-negative symmetric matrix defined by
\begin{equation}\label{11}
D_0=\int_S \alpha_\infty(\theta)\left[z+\theta\overline{\bf K}\right]\left[z+\theta\overline{\bf K}\right]^T Q(\d\theta,\d z).
\end{equation} 
\end{theorem}

If we denote by $k=\dim \ker D_0\geq 0$ and $\lambda_1>0$,..., $\lambda_{N-k}>0$ the nonzero eigenvalue of $\frac{1}{1+\Gamma}D_0$ then there exists an orthogonal matrix $P$ such that
\begin{equation*}
\frac{1}{1+\Gamma}D_0=P^T{\rm diag}\left(\lambda_1,\cdots, \lambda_{N-k},0,\cdots, 0\right)P.
\end{equation*}
Using this orthogonal basis one obtains an explicit formula for $\Pi(t,\d x)$ and thus the following explicit formula for $\pi=\Pi_1$:
\begin{equation*}
\pi\left({\rm d}y\right)=P_\sharp^T\left[\left[\bigotimes_{i=1}^{N-k}f_{\lambda_i}(y_i){\rm d}y_i\right]\otimes \left[\bigotimes_{i=N-k+1}^{N}\delta_0\left({\rm d}y_i\right)\right]\right].
\end{equation*}
Here, $\sharp$ denotes the usual push forward operator for Borel measures while for each
$\lambda>0$, the function $f_\lambda : \R \to \R$ denotes centred normal distribution with variance $\lambda$,
that is
\begin{equation*}
f_\lambda(y)=\frac{1}{\sqrt{2\pi\lambda}} \exp\left(-\frac{y^2}{2\lambda}\right).
\end{equation*}
In other words, the measure $\Pi_1$ is the multidimensional Gaussian law with variance-covariance matrix given by $\frac{1}{1+\Gamma}D_0$. We refer to \cite{SV} for a complete treatment of the relation between Gaussian processes and Fokker-Planck equations of the form \eqref{10}.

The proof of Theorem \ref{THEO2} will be divided in two parts. Our proof relies on the study of a – self-similar – rescaled family of solutions. In Section \ref{sec:proof}, we first identify the asymptotic behaviour of the solutions by deriving a convergence result in the sense of the Young measures. Then, in Section \ref{sec:proofTHEO2},  we strengthen this convergence by deriving suitable fractional Sobolev compactness estimates.

Before going to the proof of Theorem \ref{THEO2}, we apply it to some example of an hyperbolic equation with time delay.
More specifically consider the equation
\begin{equation}\label{hyperbolic}
\left[\frac{\partial }{\partial t}+q\frac{\partial}{\partial x}\right]u(t,x)=a\int_{[-1,0]} u(t+\theta,x)\eta(\d\theta)-b u(t,x), \;t>0,\;x\in\R,
\end{equation}
where $q\in\R\setminus\{0\}$, $a>0$ and $b\in\R$ are fixed parameters while $\eta\in \mathcal P([-1,0])$ is a given probability measure on the interval $[-1,0]$. This equation is supplemented with the initial data
$$
u(\theta,x)=u^0(\theta,x)\text{ with } u^0\in \C\left([-1,0];L_+^1(\R)\right) \text{ and }u^0(0,\cdot)\not\equiv 0.
$$
A special case of the above equation (with $\eta(\d\theta)=\delta_{-1}(\d\theta)$ and parameter conditions) has been studied by Laurent et al in \cite{15}. In this paper the authors derived conditions ensuring a parabolic behaviour for the solution of such a problem. Here we will show how our general result, namely Theorem \ref{THEO2}, may apply to \eqref{hyperbolic}, with general time delay dependence, to obtain similar results as in \cite{15}.\\
Set $v(t,x)=u(t,x+ qt)e^{a t}$ that satisfies the equation
\begin{equation*}
\frac{\partial v(t,x)}{\partial t}=a\int_{[-1,0]} e^{-b\theta} v(t+\theta,x- q\theta)\eta(\d \theta).
\end{equation*}
Set $y=y(t)=\int_{\R} v(t,x)\d x$ and observe that it satisfies the linear delay differential equation
\begin{equation}\label{DDeq}
\begin{split}
&y'(t)=a\int_{[-1,0]}e^{-b\theta}y(t+\theta)\eta(\d\theta),\;t>0,\\
&y(\theta)=y^0(\theta):=\int_{\R}u^0(\theta,x)\d x,\;\forall \theta\in [-1,0].
\end{split}
\end{equation} 
Since $u^0(0,\cdot)\not\equiv 0$ then $y^0\in \C\left([-1,0];\R^+\right)$ with $y^0(0)\neq 0$ and $y(t)>0$ for all $t>0$.
Hence, for $t>0$, the function $w(t,x)=\frac{v(t,x)}{y(t)}$ satisfies
\begin{equation*}
\begin{split}
\frac{\partial w(t,x)}{\partial t}=&a\int_{[-1,0]} \frac{y(t+\theta)}{y(t)}e^{-b\theta} \left[w(t+\theta,x- q\theta)-w(t,x)\right]\\
=&\int_S \alpha(t,\theta)\left[w(t+\theta,x-y)-w(t,x)\right]Q(d\theta,\d y),
\end{split}
\end{equation*}
with
\begin{equation*}
\alpha(t,\theta)=a e^{-b\theta}\frac{y(t+\theta)}{y(t)}\text{ and }Q(\d\theta,\d y)=\eta(\d\theta)\otimes \delta_{ q\theta}(\d y)\in \mathcal P\left([-1,0]\times \R\right).
\end{equation*}
As a consequence Problem \eqref{hyperbolic} re-writes as \eqref{eq-retard}. To conclude this section, it remains to check that Assumption \ref{ASS1} is satisfied for this specific example.
To that aim, one may first notice that Assumption \ref{ASS1} $(ii)$ is satisfied. And, Assumption \ref{ASS1} $(i)$ also holds true with the limit function $a_\infty=a_\infty(\theta)$ defined by
$
\alpha_\infty(\theta)=a e^{-b\theta}e^{\gamma\theta},
$ 
where $\gamma>0$ is the unique real solution of the equation
$$
\gamma=a\int_{[-1,0]}e^{(\gamma-b)\theta}\eta(\d\theta).
$$
This latter property will be checked in Appendix B.
This example will be further developed in the next section. 

\section{A probabilistic representation for some non-local diffusion equations with delay}\label{sec:proba}

In this section, we construct a $\R^N$-valued c\`adl\`ag process whose law satisfies the equation \eqref{eq-retard}. Let $(\Omega,\mathcal{A},\mathbb{P})$ be a probability space on which is built a nonhomogeneous Poisson point process $(N(t))_{t\in\mathbb{R}_+}$ with intensity function given for $t\in\mathbb{R}_+$ by 
$$
\lambda(t)=\int_S \alpha(t,\theta)Q(\d\theta,\d z)
$$
and corresponding intensity measure given by $\Lambda((0,t])=\int_0^t \lambda(s)\d s$. As usual, the sequence of times associated to the Poisson point process is defined by $T_0=0$ and, for $n\in\mathbb{N}$, by $T_n=\inf\{t\geq0 ; N(t)\geq n\}$. Notice that Assumption \ref{ASS1} implies that the intensity function stays positive and converges towards $\lambda_\infty=\int_S \alpha_\infty(\theta)Q(\d\theta,\d z)>0$ such that we do have that the $T_n$'s form a increasing sequence going to infinity, $\mathbb{P}$-a.s.

We also define, on the same probability space, a sequence $(\Theta_n,Z_n)_{n\in\mathbb{N}}$ of $S$-valued random variables such that conditional on $(T_n)_{n\in\mathbb{N}}$, the sequence $(\Theta_n,Z_n)_{n\in\mathbb{N}}$ is a sequence of independent random variables with law given, for $n\in\mathbb{N}$, by 
$$
q_{T_n}(\d\theta,\d z)=\frac{1}{\lambda(T_n)}\alpha(T_n,\theta)Q(\d\theta,\d z).
$$

Let also $(U(\theta))_{\theta\in[-1,0]}$ be a $\R^N$-valued c\`adl\`ag process such that for any $\theta\in[-1,0]$, $U(\theta)$ has for law $u^0(\theta,dx)$. We are now ready to define a $\R^N$-valued process $X$ with suitable law:
\begin{enumerate}
\item For $\theta\in[-1,0]$, we set $X(\theta)=U(\theta)$ such that $X(T_0)=U(0)$;
\item For $t\in[T_0,T_1)$, $X(t)=X(T_0)$ and then, at time $T_1$,
$$
X(T_1)=X(T_1^-+\Theta_1)+Z_1.
$$ 
\item And so on: for $n\in\mathbb{N}$, assume that the process $X$ is built up to the time $T_n$, then, for $t\in[T_n,T_{n+1})$, $X(t)=X_{T_n}$ and $X(T_{n+1})=X(T_{n+1}^-+\Theta_{n+1})+Z_ {n+1}$.
\end{enumerate}
By construction, the process $X$ is piecewise constant and c\`adl\`ag on $\mathbb{R}_+$.  In the following proposition, we verify that the law of $X$ satisfies the evolution equation \eqref{eq-retard}. In this setting, one obtains the following proposition.

\begin{proposition}
Let $f\in C_b(\R^N)$ be given. The function $t\mapsto \EE(f(X(t))$ is continuously differentiable on $(0,\infty)$ and we have, for any $t>0$,
\begin{equation}\label{eq:law}
\frac{\d}{\d t}\EE(f(X(t)))=\lambda(t)\int_{S}[\EE(f(X(t+\theta)+z))-\EE(f(X(t)))]q_t(\d \theta,\d z).
\end{equation}
\end{proposition} 

\begin{proof}
We use the decomposition of the process induced by the sequence of times $(T_n)_{n\in\mathbb{N}}$. For any $t\geq0$, we get
$$
\EE(f(X(t)))=\sum_{i=0}^\infty \EE(1_{[T_i,T_{i+1})}(t)f(X(T_i))).
$$
For any integer $i$, by conditioning on the history up to the time $T_i$, we have
$$
\EE(1_{[T_i,T_{i+1})}(t)f(X_{T_i}))= \EE\left(1_{T_i\leq t}e^{-\Lambda((T_i,t])}f(X(T_i))\right).
$$
Since $T_0=0$ is deterministic, we distinguishes this time and then use the law of $T_n$ for $n\in\mathbb{N}$ to write:
\begin{align*}
\EE(f(X(t)))=&~e^{-\int_0^t\lambda(s)\d s}\EE(f(U(0)))+\EE\left(1_{T_i\leq t}e^{-\Lambda((T_i,t])}f(X(T_i))\right)\\
=&~e^{-\Lambda((0,t])}\EE(f(U(0)))\\
&+\sum_{i=1}^\infty\int_0^t \hspace{-0.1cm}e^{-\Lambda((s,t])}\EE(f(X(T_i))|T_i=s)\lambda(s)\frac{\Lambda((0,s])^{i-1}}{(i-1)!}e^{-\Lambda((0,s])}\d s.
\end{align*}
Using the fact that $X(T_i)=X(T_i^-+\Theta_i)+Z_i$, we obtain
\begin{align*}
\EE(f(X(t)))=&~e^{-\Lambda((0,t])}\EE(f(U(0)))\\
&+\sum_{i=1}^\infty\int_0^t e^{-\Lambda((s,t])}\int_{S}\EE(f(X(s+\theta)+z))q_s(\d \theta,\d z)\\
&\hspace{5cm}\lambda(s)\frac{\Lambda((0,s])^{i-1}}{(i-1)!}e^{-\Lambda((0,s])}\d s.
\end{align*}
Finally, by summation (using the fact that $f$ is bounded),
\begin{align*}
\EE(f(X(t)))=&~e^{-\Lambda((0,t])}\EE(f(U(0)))\\
&+\int_0^t e^{-\Lambda((s,t])}\lambda(s)\int_{S}\EE(f(X(s+\theta)+z))q_s(\d \theta,\d z)\d s.
\end{align*}
The result follows by derivation.
\end{proof}

{Denoting by $\mu(t,\d x)$ the law of $X(t)$, equations \eqref{eq:law} and \eqref{eq-retard} are equivalent. Thus, the process $X$ gives a probabilistic representation of the evolution equation studied in the present paper.} Notice that in general, $X$ is not a Markov process: the variables $\Theta_n$'s introduce memory and model the delay in the evolution equation \eqref{eq-retard}. {When there is no delay, that is when $Q(\d \theta,\d z)=\delta_0(\d \theta)Q(\d z)$, the studied process is simply an inhomogeneous continuous time Markov chain, as already mentioned in the introduction.} Let us also notice that this is always possible to recast a non-Markovian process into a Markovian one: for processes very similar to the process $X$ under consideration, this has been done in \cite{Meerschaert}. However, the fact that the evolution of the law of $X$ satisfies equation \eqref{eq-retard} allows for a more direct approach. 

Theorem \ref{THEO2} is the central limit theorem associated to the process $X$. Its proof will be the object of Section \ref{sec:proof}. In probabilistic term, Theorem \ref{THEO2} reads as follows.

\begin{theorem}
The process $(\sqrt t (X(t)/t -{\bf K}))_{t\in(0,\infty)}$ converges in law towards a centred Gaussian random variable with variance-covariance matrix given by 
$$
\mathbf{Q}=\frac{\lambda_\infty}{1-\lambda_\infty\EE(\Theta_\infty)}\EE((Z_\infty+\Theta_\infty {\bf K})(Z_\infty+\Theta_\infty {\bf K})^T),
$$
where the couple $(\Theta_\infty,Z_\infty)$ has for law $\alpha_\infty(\theta)Q(\d \theta,\d z)/\lambda_\infty$.
\end{theorem}

Of course, this central limit theorem implies that  the process $(X(t)/t)_{t\in(0,\infty)}$ converges in probability towards ${\bf K}$. We can strengthen this convergence using the explicit construction of the process and obtain the following strong law of large numbers.

\begin{theorem}
Assume the result of Theorem \ref{THEO2}. Then, the process $(X(t)/t)_{t\in(0,\infty)}$ converges almost-surely towards ${\bf  K}=\frac{\lambda_\infty\EE(Z_\infty)}{1-\lambda_\infty\EE(\Theta_\infty)}$ where the couple $(\Theta_\infty,Z_\infty)$ has for law $\alpha_\infty(\theta)Q(\d \theta,\d z)/\lambda_\infty$.
\end{theorem}

\begin{proof}
First, notice that for any $t>0$,
$$
|X(t)|\leq \max_{\theta\in[-1,0]}|U(\theta)|+\sum_{i=1}^{N(t)} |Z_i|.
$$
Therefore,
$$
\frac{|X(t)|}{t}\leq \frac1t\max_{\theta\in[-1,0]}|U(\theta)|+\frac{N(t)}{t}\frac{1}{N(t)}\sum_{i=1}^{N(t)} |Z_i|.
$$
We have, almost-surely,
$$
\lim_{t\to\infty}\frac{N(t)}{t}=\lim_{t\to\infty}\frac{\Lambda((0,t])}{t}=\lambda_\infty.
$$
Thanks to Assumption \ref{ASS1} and the conditional independence of the $Z_i$'s, we can use the conditional law of large number, see for instance \cite{Majerek} and obtain that we almost-surely have,
$$
\lim_{t\to\infty}\frac{1}{N(t)}\sum_{i=1}^{N(t)} |Z_i|=\frac{1}{\lambda_\infty}\int_S |z| \alpha_\infty(\theta) Q(\d \theta, \d z).
$$
The process $(X(t)/t)_{t\in(0,\infty)}$ is thus almost-surely asymptotically bounded. Therefore, its convergence in probability towards the deterministic value ${\bf K}$ implies its almost-sure convergence towards this same value.
\end{proof}

Remark that since the process $(X(t)/t)_{t\in(0,\infty)}$ is bounded for $t$ large enough, we can use a slight refinement of equation \eqref{eq:law} in order to include an explicit dependence in time and obtain, for $y(t)=\EE(X(t)/t)$ with $t>0$,
\begin{align*}
\frac{\d}{\d t}y(t)=&~\lambda(t)\int_S [y(t+\theta)-y(t)]\left[1+\frac\theta t\right] q_t(\d \theta,\d z)\\
&+\frac{1}{t}\left[-y(t)\left(1-\lambda(t)\int_S \theta q_t(\d \theta,\d z)\right)+\lambda(t)\int_S zq_t(\d \theta,\d z)\right].
\end{align*}
Notice that since ${\bf K}$ is the limit of $(X(t)/t,t>0)$ then, by dominated convergence, ${\bf K}$ is also the limit of $(y(t),t>0)$, a fact that is, at first glance, far to be obvious if you only look at the above differential equation with delay. 

Let us give a first illustration of the two above results by considering the case where $X(\theta)=0$ for all $\theta\in[-1,0]$ and
$$
\alpha(t,\theta)=1,\qquad Q(\d \theta,\d z)=\delta_{0}(\d\theta)\delta_{1}(\d z).
$$
Then $X$ is the classical Poisson process satisfying
$$
\lim_{t\to\infty}\frac{X(t)}{t}=1\text{ a.s.},\qquad \lim_{t\to\infty}\sqrt{t}\left(\frac{X(t)}{t}-1\right)=\mathcal{N}(0,1)\text{ in law,}
$$
where $\mathcal{N}(0,1)$ denotes the centred gaussian law with variance $1$. Now, with a one unit time delay, namely with
$$
\alpha(t,\theta)=1,\qquad Q(\d \theta,\d z)=\delta_{-1}(\d\theta)\delta_{1}(\d z),
$$
our results give
$$
\lim_{t\to\infty}\frac{X(t)}{t}=\frac12\text{ a.s.},\qquad \lim_{t\to\infty}\sqrt{t}\left(\frac{X(t)}{t}-\frac12\right)=\mathcal{N}\left(0,\frac18\right)\text{ in law.}
$$
In this example, we see that the delay slow down the process and reduces its scattering.

{As a second illustration, we consider the stochastic process associated to the hyperbolic equation with delay (\ref{hyperbolic}) considered at the end of the previous section. In this setting, for $t>0$ and $\theta\in[-1,0]$, the parameters are
\begin{equation}\label{hdde_parameters}
\alpha(t,\theta)=a e^{-b\theta}\frac{y(t+\theta)}{y(t)}\text{ and }Q(\d\theta,\d y)=\eta(\d\theta)\otimes \delta_{ q\theta}(\d y)\in \mathcal P\left([-1,0]\times \R\right),
\end{equation}
such that
$$
\lambda(t)=\int_{[-1,0]\times \R} \alpha(t,\theta)Q(\d\theta,\d z)=\int_{[-1,0]}a e^{-b\theta}\frac{y(t+\theta)}{y(t)}\eta(\d\theta).
$$
For our simulation purpose, as in Laurent et al in \cite{15}, we set $\eta(\d \theta)=\delta_{-1}(\d\theta)$. In this case,
$$
\lambda(t)=a e^{b}\frac{y(t-1)}{y(t)},
$$
where
\begin{equation}\label{DDeqLaurent}
\begin{split}
&y'(t)=ae^{b}y(t-1),\;t>0,\\
&y(\theta)=y^0(\theta):=\int_{\R}u^0(\theta,x)\d x,\;\forall \theta\in [-1,0],
\end{split}
\end{equation}
such that the intensity measure is given by $\Lambda((0,t])={\rm log}(y(t)/y(0))$. In Figure \ref{examplehdde}, the law of large numbers as well as the associated central limit theorem are illustrated for this process. We also display, in Figure \ref{exampleddebis}, another illustration of these two limit theorems with $Q(\d \theta,\d y)=\delta_{-1}(\d \theta)\otimes \mathbf{1}_{[-1/2,1/2]}(y)\d y$, all other things being equal.}

\begin{figure}
\begin{center}
\includegraphics[width=5.5cm]{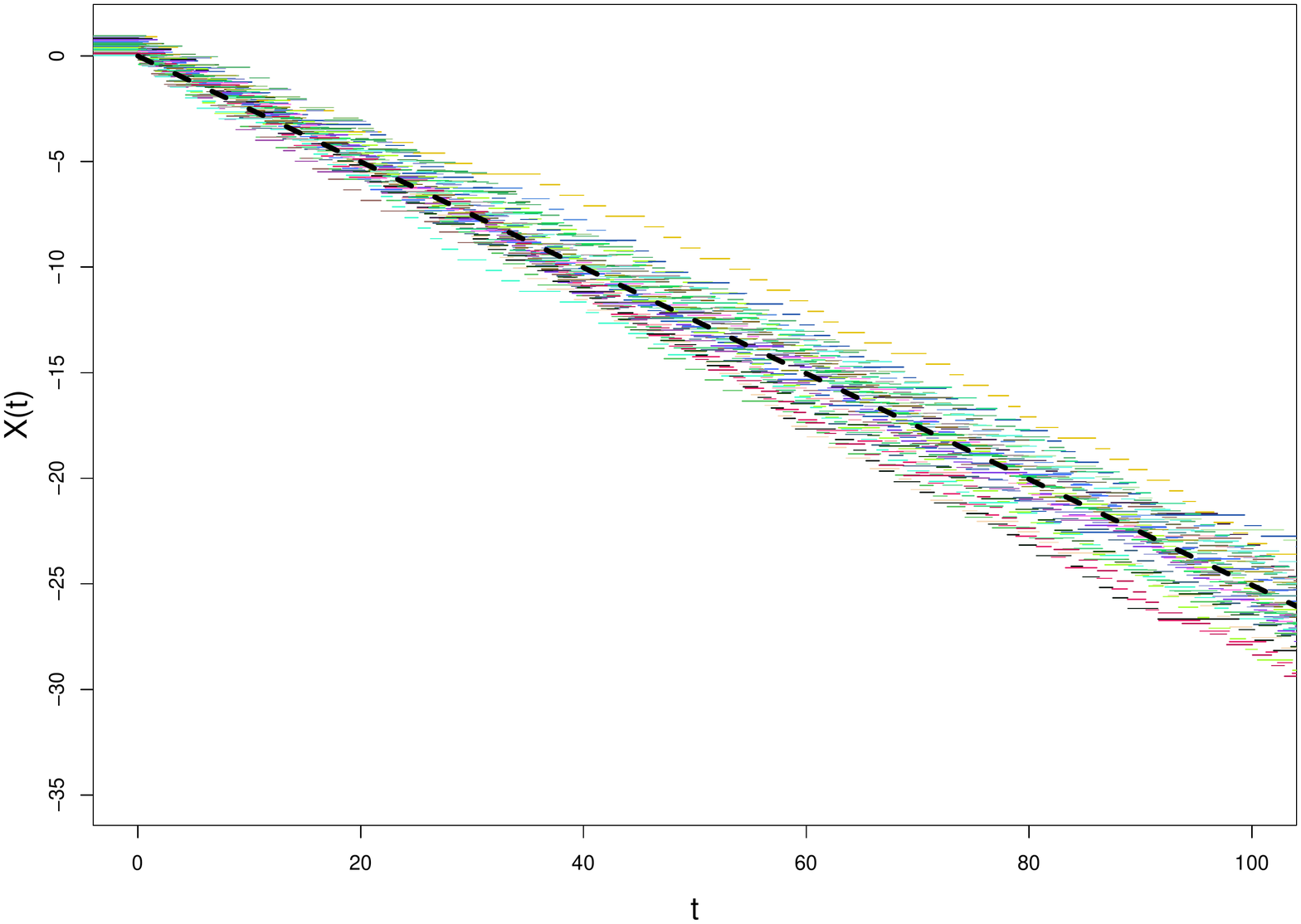}\hspace{1cm}\includegraphics[width=5.5cm]{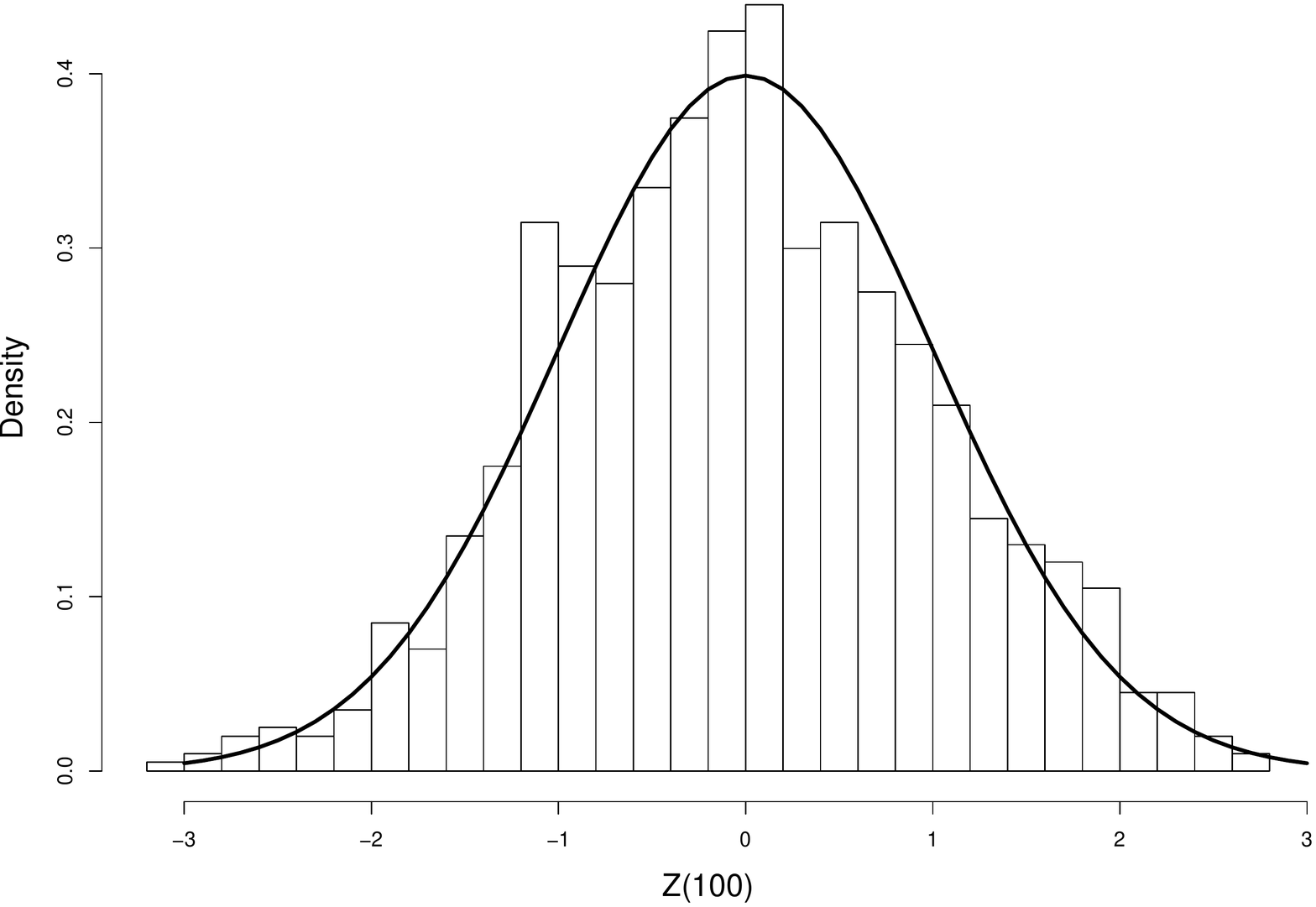}
\caption{{Left: 40 trajectories of $X$ (one color per trajectory) with the parameters defined in (\ref{hdde_parameters}) with $a=1.01$, $b=1$, $q=1$ and initial condition $u^0(\theta,\d x)=\mathbf{1}_{[0,1]}(x)\d x$ (uniform law on $[0,1]$). The dashed line has slope $\mathbf{K}\simeq -0.25$. Right: distribution of $Z(100)=\sqrt{100}\frac{X(100)/100-K}{{\sqrt\mathbf{Q}}}$ (${\sqrt{\mathbf{ Q}}}\simeq0.18$), obtained from $1000$ trajectories, compared to the density of the normal law.}}\label{examplehdde}
\end{center}
\end{figure}

\begin{figure}
\begin{center}
\includegraphics[width=5.5cm]{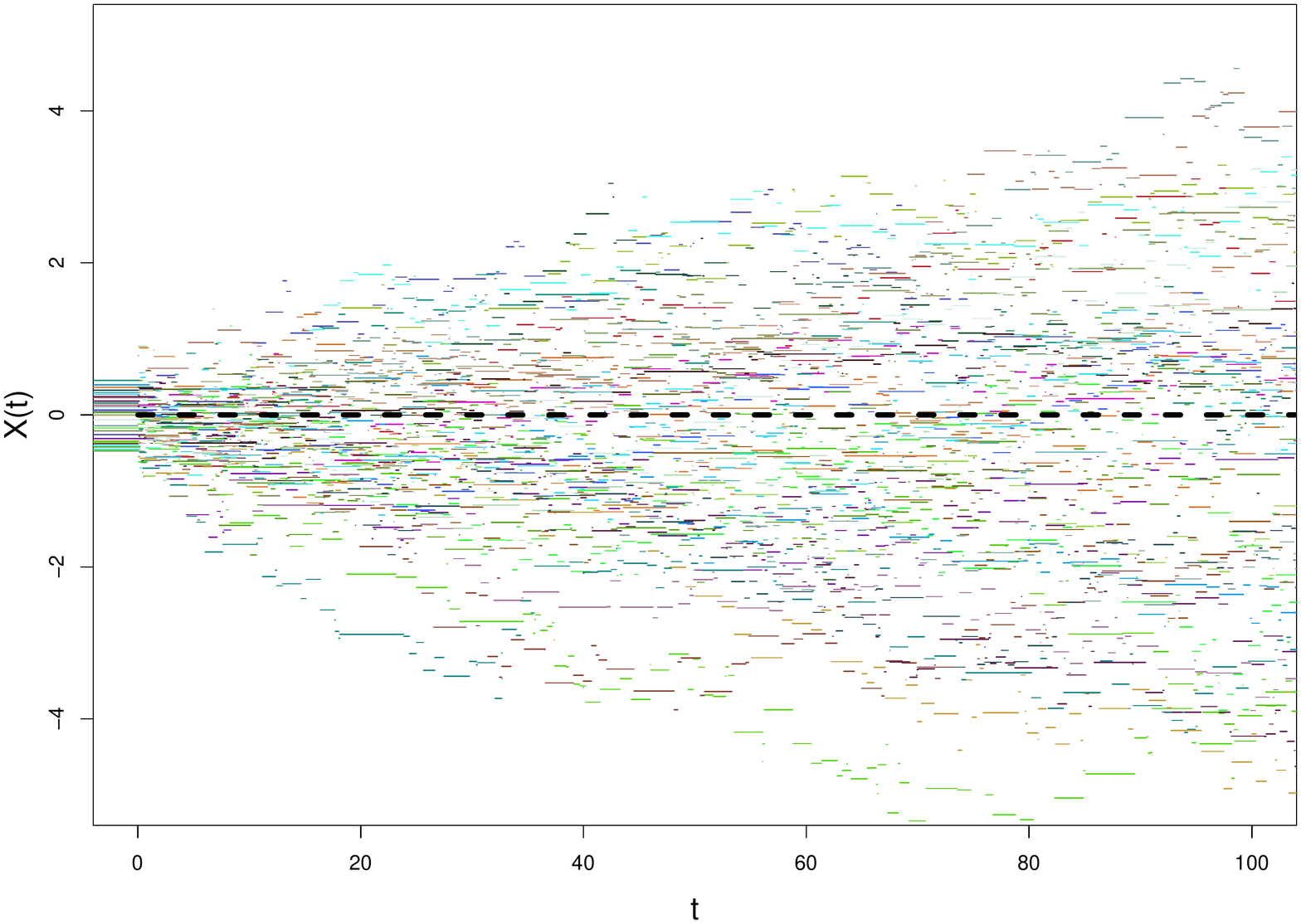}\hspace{1cm}\includegraphics[width=5.5cm]{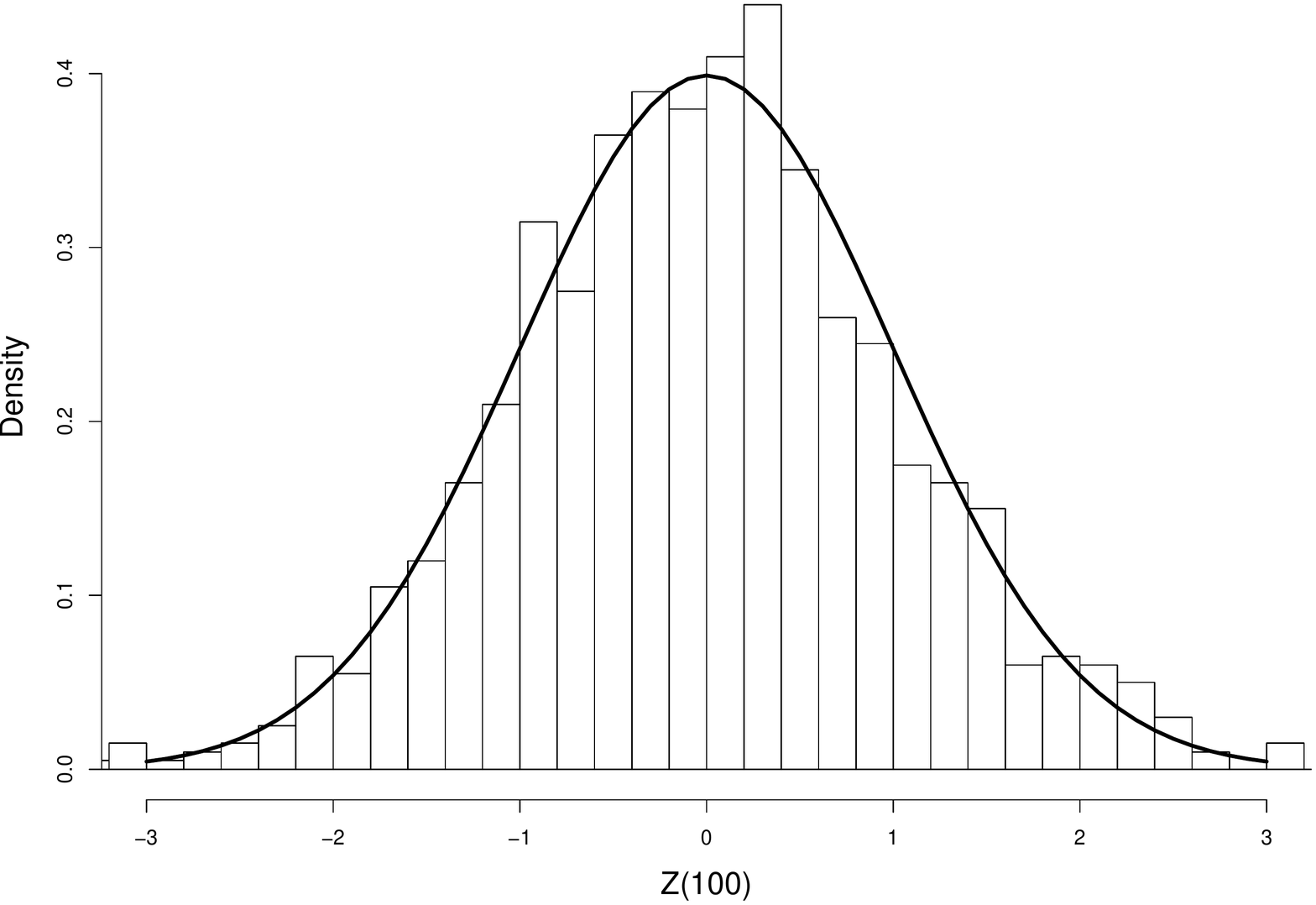}
\caption{{Left: 40 trajectories of $X$ (one color per trajectory) with the parameters defined in (\ref{hdde_parameters}) with $a=1.01$, $b=1$, $q=1$ but $Q(\d \theta,\d y)=\delta_{-1}(\d \theta)\otimes \mathbf{1}_{[-1/2,1/2]}(y)\d y$ and initial condition $u^0(\theta,\d x)=\mathbf{1}_{[-1/2,1/2]}(x)\d x$ (uniform law on $[-1/2,1/2]$). The dashed line has slope $\mathbf{K}=0$. Right: distribution of $Z(100)=\sqrt{100}\frac{X(100)/100-K}{\sqrt\mathbf{Q}}$ (${\sqrt\mathbf{Q}}\simeq0.20$), obtained from $1000$ trajectories, compared to the density of the normal law.}}\label{exampleddebis}
\end{center}
\end{figure}


\section{Self-similar rescaling and limit identification}\label{sec:proof}

In this section, we deal with rescaled solution and we shall prove that such
a family of rescaled solution converges toward the solution of the possibly
degenerate parabolic equation \eqref{10} for a suitable topology.

Before going further let us introduce some notations that will be used in the
sequel.
We define $\Gamma>0$ and the vector ${\bf K}\in\R^N$ by
\begin{equation}\label{def-K-Gamma}
\Gamma:=\int_S (-\theta)\alpha_\infty(\theta) Q(\d\theta,\d z)\text{ and }{\bf K}:=\frac{1}{1+\Gamma}\int_S \alpha_\infty (\theta) z Q(\d\theta,\d z). 
\end{equation}
\begin{remark}\label{REM-identity}
Note that the vector ${\bf K}$ satisfies the following identity, that will be used latter,
\begin{equation*}
\int_S \alpha_\infty(\theta)\left[z+\theta {\bf K}\right]Q(\d\theta,\d z)={\bf K}.
\end{equation*}
\end{remark}

We also define the function $H=H(t):[-1,\infty)\to \R^N$ by
\begin{equation}\label{25}
H(t)=-\frac{1}{1+\Gamma}\int_0^t\int_S \alpha(t,\theta) x Q(\d\theta,\d x) \d s,\;\forall t\geq -1.
\end{equation}

As mentioned above we aim at investigating the large behaviour of $t \mapsto\mu(t,\d x)$.
Thus we need to shift this function in order to follow the mass and to not lose its transportation. As
it will appear clearly in the proofs presented below, a suitable shift is to follow
the solution along the path $t\mapsto H(t)$ introduced above in \eqref{25}.
Hence we translate the measure solution $t\mapsto \mu(t,\d x)$ by introducing the measure
\begin{equation}\label{27}
\nu(t,\cdot)=\delta_{H(t)}\ast \mu(t,\cdot)\in\P,\;\forall t\geq -1.
\end{equation}
In order to be more precise, the above convolution product, $\ast$, means that for all $\varphi\in\C_b(\R^N)$:
\begin{equation*}
\int_{\R^N}\varphi(x)\nu(t,\d x)=\int_{\R^N}\varphi(x+H(t))\mu(t,\d x),\;\forall t\geq -1.
\end{equation*}
Notice that with the notations of Section \ref{sec:proba}, $\nu(t,\d x)$ is the law of the process $X+H$ at time $t$.

The following lemma holds true.
\begin{lemma}\label{LE17}
The map $t\mapsto \nu (t,\d x)$ is continuous from $[-1,\infty)$ into $\P$.
It furthermore satisfies, for each test function $\phi\in W^{1,1}\left(0,\infty;\C_b(\R^N)\right)\cap L^1\left(0,\infty;\C_b^1(\R^N)\right)$, the equation
\begin{equation*}
\begin{split}
&\int_{\R^+\times\R^N}\left(\partial_t+H'(t)\nabla\right)\phi(t,x)\,\nu(t,\d x)\d t=-\int_{\R^N}\phi(0,x)\nu(0,\d x)\\
&+\int_0^\infty\int_{\R^N}\left[\int_S\alpha(t,\theta)Q(\d\theta,\d z)\right]\phi(t,x)\nu(t,\d x)\d t\\
&-\int_0^\infty \int_{S}\alpha(t,\theta)\int_{\R^N}\phi\left(t,x+z+G(t,\theta)\right)\nu(t+\theta,\d x)Q(\d \theta,\d z)\d t,
\end{split}
\end{equation*}
wherein we have set $G(t,\theta)=H(t)-H(t+\theta)$.
\end{lemma}

\begin{proof}
Let us first note that since the map $t \mapsto H(t)$ is continuous, it follows from Portemanteau theorem (see Theorem \ref{THEO4} in Appendix A) that the map $t \mapsto\nu(t,\d x)$ is also continuous.

Let $\phi=\phi(t,x)\in  W^{1,1}\left(0,\infty;\C_b(\R^N)\right)\cap L^1\left(0,\infty;\C_b^1(\R^N)\right)$ be a given test function. Then, one has:
\begin{equation*}
\begin{split}
&\int_{\R^+\times\R^N}\left(\partial_t+H'(t)\nabla\right)\phi(t,x)\nu(t,\d x)\d t\\
=&\int_0^\infty\int_{\R^N}\left(\partial_t+H'(t)\nabla\right)\phi(t,x+H(t))\mu(t,\d x)\d t\\
=&\int_0^\infty \int_{\R^N}\partial_t \psi(t,x)\mu(t,\d x)\d t,
\end{split}
\end{equation*}
wherein we have set $\psi(t,x):=\phi(t,x+H(t))$. It is worth noticing that $\psi\in W^{1,1}\left(0,\infty;\C_b(\R^N)\right)\cap L^1\left(0,\infty;\C_b(\R^N)\right)$. Then, integrating by parts yields
\begin{equation*}
\begin{split}
&\int_{\R^+\times\R^N}\left(\partial_t+H'(t)\nabla\right)\phi(t,x)\nu(t,\d x)\d t\\
&=-\int_{\R^N}\phi(0,x)\nu(0,\d x)+\int_0^\infty\int_{\R^N}\left[\int_S\alpha(t,\theta)Q(\theta,\d z)\right]\phi(t,x)\nu(t,\d x)\d t\\
&-\int_0^\infty \int_{S}\alpha(t,\theta)\int_{\R^N}\phi\left(t,x+z+H(t)\right)\mu(t+\theta,\d x)Q(\d \theta,\d z)\d t.
\end{split}
\end{equation*}
On the other hand one has
\begin{equation*}
\begin{split}
&\int_0^\infty \int_{S}\alpha(t,\theta)\int_{\R^N}\phi(t,x+z+H(t))\mu(t+\theta,\d x)Q(\d \theta,\d z)\d t\\
=&\int_0^\infty \int_{S}\alpha(t,\theta)\int_{\R^N}\phi\left(t,x+z+G(t,\theta)\right)\nu(t+\theta,\d x)Q(\d \theta,\d z)\d t.
\end{split}
\end{equation*}
The result follows by coupling the above computations.
\end{proof}

In order to derive the asymptotic self-similar behaviour of $t\mapsto \nu(t,\d x)$ for $t>\!\!>1$ we are interested
in the asymptotic as $\lambda\to\infty$ of the rescaled family of probability measures $p^\lambda(t,\d x)$ defined, for $t \geq -\lambda^{-2}$  and for each test function $\varphi\in\C_b(\R^N)$ by
\begin{equation}\label{def-p-lambda}
\int_{\R^N}\varphi(x) p^\lambda(t,\d x)=\int_{\R^N}\varphi\left(\frac{x}{\lambda}\right)\nu(\lambda^2 t,\d x),\;\forall t\geq -\lambda^{-2}.
\end{equation}
To study its behaviour for $\lambda>\!\!>1$, we first write down the equation satisfied by $p^\lambda$ by noticing that
\begin{equation*}
\int_{\R^N}\varphi(x)\nu(t,\d x)=\int_{\R^N}\varphi(\lambda x)p^\lambda (\lambda^{-2} t,\d x),\;\forall t\geq -1,\;\forall \varphi\in \C_b(\R^N).
\end{equation*}
Notice that with the notations of Section \ref{sec:proba}, $p^\lambda(t,\d x)$ is the law of the rescaled process $\frac{X(\lambda^2 t)+H(\lambda^2 t)}{\lambda}$.

The equation for $p^\lambda$ is expressed in the next lemma: 
\begin{lemma}\label{LE18}
Let $\lambda>0$ be given. Then the function $t\mapsto p^\lambda(t,\d x)$ satisfies, for each test function $\phi\in W^{1,1}\left(0,\infty;\C_b(\R^N)\right)\cap L^1\left(0,\infty;\C_b^1(\R^N)\right)$,
\begin{equation}\label{28}
\begin{split}
&\int_{\R^+\times\R^N}\left(\partial_t+\lambda H'\left(\lambda^2t\right)\nabla\right)\phi p^\lambda(t,\d x)\d t=-\int_{\R^N}\phi(0,x)p^\lambda(0,\d x)\\
&+\lambda^2\int_0^\infty \int_{S\times \R^N}\alpha(\lambda^2 t,\theta)\phi(t,x)p^\lambda (t,\d x)Q(\d \theta,\d z)\d t\\
&-\lambda^2 \int_0^\infty \int_{S\times \R^N}\alpha(\lambda^2 t,\theta)\phi(t,x+z) p^\lambda\left(t+\frac{\theta}{\lambda^2},\d x\right) Q^\lambda (\lambda^2 t,\theta;\d\theta,\d z)\d t\\
\end{split}
\end{equation}
wherein, for each $\lambda>0$, the measure $Q^\lambda (t,\theta;\d\theta,\d z)$ is defined, for each $t\geq -\lambda^{-2}$, for each $\varphi=\varphi(\theta,z)\in \mathcal C_b(S)$, by
\begin{equation*}
\int_S \varphi(\theta,z)Q^\lambda (t,\theta;\d\theta,\d z)=\int_S \varphi\left(\theta,\frac{z}{\lambda}+\frac{1}{\lambda}G(t,\theta)\right)Q(\d\theta,\d z).
\end{equation*}
\end{lemma}

The proof of the derivation of this equation directly follows from Lemma
\ref{LE17} using the definition of the map $t\mapsto p^\lambda(t,\d x)$. 

Here, one may note that we have used a parabolic scaling but the equation is not invariant under this scaling.
We are now interested in the asymptotic behaviour of $p^\lambda$ as $\lambda\to\infty$ and the
main result of this section reads as follows:
\begin{theorem}[Young measure convergence]\label{THEO19}
Let Assumption \ref{ASS1} be satisfied. Recalling the definition of $\Gamma$ in \eqref{9} and of $D_0$ in \eqref{11}, the following convergence holds true for each test function $f\in L^1\left(0,\infty;\C_b(\R^N)\right)$:
\begin{equation*}
\lim_{\lambda\to\infty}\int_{\R^+\times\R^N} f(t,x)p^\lambda(t,\d x)\d t=\int_{\R^+\times\R^N}f(t,x)\Pi(t,\d x)\d t,
\end{equation*}
where $\Pi=\Pi(t,\d x)\in \C\left([0,\infty);\P\right)$ is the solution of
\begin{equation*}
\partial_t \Pi-\frac12{\rm div}\,\left(\frac{1}{1+\Gamma}D_0\nabla\Pi\right)=\delta_{t=0}\otimes \delta_0. 
\end{equation*}
\end{theorem}
This results ensures the convergence of $p^\lambda$ to $\Pi$ in a rather weak sense, especially in time. This convergence is not sufficient to prove Theorem \ref{THEO2}. It will be strengthen in the next section to obtain time pointwized convergence.
The proof of the above result will make use of Young measure theory. We refer to reader to \cite{Castaing, Valadier} for more details and also to Appendix A where basic properties of Young measures are recalled.

Note also that if we choose $f(t,x)=f(x)f_T(t)$ where $f$ is continuous and bounded and $f_T$ is the density of some positive random variable $T$, the above result implies that the measure $\int_0^\infty p^\lambda(t,\d x)f_T(t)\d t$ converges in law towards the measure $\int_0^\infty\Pi(t,\d x)f_T(t)\d t$. If we denote by $Z_\lambda$ the rescaled process $\left(\frac{X(\lambda^2 t)+H(\lambda^2 t)}{\lambda}\right)_{t\in\R+}$, one interpretation is to say that $Z_\lambda(T)$, where $T$ is a random variable with density $f_T$ independent of $Z_\lambda$, converges in law when $\lambda$ goes to infinity towards a random variable with law $\int_0^\infty\Pi(t,\d x)f_T(t)\d t$. As mentioned above, the objective of Section \ref{sec:proofTHEO2} will be to replace the random variable $T$ by the deterministic time $1$. The above theorem also tells us that, in the sense of Young measures, the stochastic process $Z_\lambda$ converges towards a $N$-dimensional centred Wiener process with variance-covariance matrix given by $\frac{1}{1+\Gamma}D_0$. 

In order to prove Theorem \ref{THEO19}, we fix a sequence $\{\lambda_n\}_{n\geq 0}\subset (0,\infty)$ such that
\begin{equation}\label{29}
\lim_{n\to\infty}\lambda_n=\infty,
\end{equation}
and we consider the sequence of Young measures 
$$\left\{t\mapsto p^{\lambda_n}(t,\d x)\right\}_{n\geq 0}\subset L_{\omega*}^\infty(0,\infty;\P).
$$
Then, denoting by $\mathcal M^+$ the set of positive Borel measures on $\R^N$, due to Lemma \ref{LE6} there exists a subsequence, still denoted by $\lambda_n$, and $P\equiv P(t,\d x)\in L^\infty_{\omega*}\left(0,\infty;\M^+\right)$ such that
\begin{equation*}
P(t,\d x)\in \M^+\text{ and }\int_{\R^N}P(t,\d x)\leq 1\;a.e.\;t\geq 0,
\end{equation*}
and
\begin{equation}\label{30}
\lim_{n\to\infty} p^{\lambda_n}=P\text{ weakly$*$ in }\left(L^1(0,\infty;\C_b(\R^N))\right)'.
\end{equation}
Above, $L^\infty_{\omega*}\left(0,\infty;\M^+\right)$ denotes the set of weakly$∗$ measurable maps from $\R^+$ into $\M^+$ and that are essentially bounded (see also Appendix A for more details).

In our next proposition, we will identify the limit measure $P=P(t,\d x)$. 
\begin{proposition}\label{PROP20}
Let Assumption \ref{ASS1} be satisfied. Then the function $P\equiv P(t,\d x)$ is in  $L^\infty_{\omega*}\left(0,\infty;\M^+\right)$ defined above is the unique tempered distribution solution of \eqref{10}, that is $P(t,\d x)\equiv \Pi(t,\d x)$.
\end{proposition}

Since $L^1 \left(0,\infty;\C_0(\R^N)\right)$ is separable, the balls of its dual space are metrizable. This implies the following corollary:
\begin{corollary}\label{CORO21}
Under the assumptions of Proposition \ref{PROP20}, one obtains 
\begin{equation*}
\lim_{\lambda\to\infty} p^\lambda=\Pi(t,\d x)\text{ weakly$*$ in }\left(L^1 \left(0,\infty;\C_0(\R^N)\right)\right)'.
\end{equation*}
\end{corollary}
Now note that one has
\begin{equation*}
\int_{\R^N}\Pi(t,x)\d x=1\;a.e.\;t\geq 0.
\end{equation*}
This mass conservation property ensures that the family $\{p_\lambda\}_{\lambda>0}$ is a tight family of Young measures. Hence Lemma \ref{LE7} applies and leads us to the following corollary
\begin{corollary}\label{CORO22}
Under the assumptions of Proposition \ref{PROP20}, for each test function $f\in L^1(0,\infty;\C_b(\R^N))$ one has
\begin{equation*}
\lim_{\lambda\to\infty} \int_{\R^+\times\R^N}f(t,x)p^\lambda(t,\d x)\d t=\int_{\R^+\times\R^N} f(t,x)\Pi(t,\d x)\d t.
\end{equation*}
\end{corollary}

To conclude these remarks, one has obtained that proving Proposition \ref{PROP20} is
sufficient to complete the proof of Theorem \ref{THEO19}. Thus in the remaining of this section we shall focus on proving Proposition \ref{PROP20}.

In order to prove Proposition \ref{PROP20}, we shall first derive preliminary lemma.
In the sequel the notation $\mathcal S$ will be used to denote the Schwartz space, that is
the set of rapidly decreasing functions, while $\mathcal S'$ will be used to denote its dual
space, the space of tempered distributions.\\
Next, using these notations, our first lemma reads as follows.
\begin{lemma}\label{LE23}
Let $F\in \mathcal C(\R^+)\cap L^1(\R^+)$ be given. Then the following convergence holds true:
\begin{equation*}
\lim_{\lambda\to\infty}\lambda^2F(\lambda^2 t)\left(p^\lambda(t,\d x)\otimes \d t\right)=\left(\int_0^\infty F(s)\d s\right)\delta_{t=0}\otimes \delta_0\text{ in }\mathcal S'(\R^+\times\R^N),
\end{equation*}
that is for each test function $\phi\in \mathcal S(\R^+\times\R^N)$ one has
\begin{equation*}
\lim_{\lambda\to\infty}\lambda^2\int_{\R^+\times\R^N}F(\lambda^2 t)\phi(t,x)p^\lambda(t,\d x)\d t=\left(\int_0^\infty F(s)\d s\right)\phi(0,0).
\end{equation*}
\end{lemma}

\begin{proof}
Let $\phi\equiv \phi(t,x)\in \mathcal S(\R^+\times\R^N)$ be given. Then from the definition of $p^\lambda$ one obtains:
\begin{equation*}
\begin{split}
\lambda^2\int_{\R^+\times\R^N}F(\lambda^2 t)\phi(t,x)&p^\lambda(t,\d x) \d t=\lambda^2\int_{\R^+\times\R^N}F(\lambda^2 t)\phi\left(t,\frac{x}{\lambda}\right)\nu(\lambda^2 t,\d x)\d t\\
=&\int_{\R^+\times\R^N}F(t)\phi\left(\frac{t}{\lambda^2},\frac{x+H(t)}{\lambda}\right)\mu(t,\d x)\d t\\
=&\int_{\R^+\times\R^N}F(t)\phi\left(\frac{t}{\lambda^2},\frac{x+H(t)}{\lambda}\right)\left(\mu(t,\d x)\otimes\d t\right).
\end{split}
\end{equation*}
Finally since function $(t, x) \mapsto F (t)$ belongs to $L^1\left(\R^+\times \R^N ; \mu(t,\d x)\otimes \d t\right)$, Lebesgue
convergence theorem applies and completes the proof of Lemma \ref{LE23}.
\end{proof}

Now equipped with this lemma we shall first prepare the equation
before passing to the limit $\lambda\to\infty$ and more precisely through the sequence $\lambda_n$
as $n\to\infty$.

\begin{lemma}[Preparation of the equation]\label{LE24}
Let $\phi\in W^{1,1}\left(0,\infty;\C_b(\R^N)\right)\cap L^1(0,\infty;\C_b^2(\R^N))$ be given. For all $\lambda>0$ the function $t\mapsto p^\lambda(t,\d x)$ satisfies
\begin{equation}\label{31}
\int_{\R^+\times\R^N}\partial_t\phi\d p^\lambda(t,\d x)\d t=-\int_{\R^N}\phi(0,x)p^\lambda(0,\d x)-\mathcal{T}_1^\lambda[\phi]-\mathcal T_2^\lambda[\phi]+\mathcal T_3^\lambda[\phi],
\end{equation}
wherein we have set
\begin{equation*}
\begin{split}
&\mathcal{T}_1^\lambda[\phi]=\lambda\int_{0}^\infty \int_{\R^N}\nabla \phi(t,x)H'(\lambda^2t)p^\lambda(t,\d x)\d t\\
&+\lambda \int_0^\infty \hspace{-0.2cm}\int_S \alpha(\lambda^2 t,\theta)\left[z+G(\lambda^2 t,\theta)\right]\left[\int_{\R^N}\nabla \phi(t,x)p^\lambda\left(t+\frac{\theta}{\lambda^2},\d x\right)\right]Q(\d\theta,\d z)\d t,
\end{split}
\end{equation*}
\begin{equation*}
\mathcal{T}_2^\lambda[\phi]=\int_0^\infty \int_{S}\int_{\R^N}\Phi^\lambda[\phi](t,\theta,x,z)p^\lambda \left(t+\frac{\theta}{\lambda^2},\d x\right)Q(\d\theta,\d z) \d t;
\end{equation*}
with $\Phi^\lambda$ defined by
\begin{align*}
&\Phi^\lambda [\phi](t,\theta,x,z)\\
&=\alpha(\lambda^2 t,\theta)\int_0^1 (1-s)D^2\phi\left(t,x+s\frac{z+G(\lambda^2t,\theta)}{\lambda}\right)\cdot [z+G(\lambda^2t,\theta)]^2\d s;
\end{align*}
and
\begin{align*}
&\mathcal{T}_3^\lambda[\phi]\\
&=\lambda^2\int_0^\infty \int_{S}\alpha(\lambda^2t,\theta)\int_{\R^N}\phi(t,x)\left[p^\lambda(t,\d x)-p^\lambda\left(t+\frac{\theta}{\lambda^2},\d x\right)\right] Q(\d \theta,\d z) \d t.
\end{align*}
\end{lemma}
The above decomposition directly follows from the formula obtained in \eqref{28}.

In order to prove Proposition \ref{PROP20} we shall now study the convergence, as $n\to \infty$, of the different terms arising in the decomposition described in the above lemma with $\lambda=\lambda_n$. Before doing so, let us
recall that, due to Assumption \ref{ASS1}, one has
\begin{equation*}
\alpha(t,\theta)=\alpha_\infty(\theta)+O\left(e^{-\beta t}\right)\text{ uniformly for $\theta\in [-1,0]$ as $t>\!\!>1$}.
\end{equation*}
As a consequence, recalling the definition of $H=H(t)$ in \eqref{25}, one obtains 
\begin{equation}\label{asympt}
\begin{split}
&H(t)=-{\bf K}t+O(1)\text{ and }H'(t)=-{\bf K}+O(e^{-\beta t})\text{ as $t\to\infty$},\\
&G(t,\theta)=\theta {\bf K}+O(e^{-\beta t})\text{ uniformly for $\theta\in [-1,0]$ as $t\to\infty$}.
\end{split}
\end{equation}
These asymptotic expansions will be extensively used in the sequel.

We are now able to investigate the asymptotic behaviour as $\lambda\to\infty$ of each term arising in the decomposition provided by Lemma \ref{LE24}.
To that aim let us fix a test function $\phi\in \mathcal S(\R^+\times\R^N)$. In the sequel of this proof we shall omit to explicitly write down the dependence with respect to $\phi$ in the decomposition of Lemma \ref{LE24}, that is for $i=1,..,3$ we shall write $\mathcal T_i^\lambda$ instead of $\mathcal T_i^\lambda[\phi]$.\\
Our convergence proof will be achieved in the next four steps.\\
\noindent{\bf Step 0:} Recalling \eqref{30} one first obtains that
\begin{equation*}
\lim_{n\to\infty}\int_{\R^+\times\R^N}\partial_t\phi p^{\lambda_n}(t,\d x)\d t=\int_{\R^+\times\R^N}\partial_t\phi P(t,\d x)\d t.
\end{equation*}
Next note that for each $\lambda>0$ one has
$$
\int_{\R^N}\phi(0,x)p^\lambda(0,\d x)=\int_{\R^N}\phi\left(0,\frac{x}{\lambda}\right)\mu(0,\d x).
$$
Hence as $\lambda\to \infty$ Lebesgue convergence theorem ensures that
$$
\lim_{\lambda\to\infty}\int_{\R^N}\phi(0,x)p^\lambda(0,\d x)=\phi\left(0,0\right).
$$
\noindent{\bf Step 1:} In this step we investigate the behaviour of $\mathcal T_1^\lambda$ as $\lambda\to\infty$. To that aim note that one has
\begin{equation*}
\begin{split}
&\lambda\int_{S}\int_{0}^\infty \alpha(\lambda^2t,\theta)\left[z+G(\lambda^2t,\theta)\right]\cdot \left\{\int_{\R^N}\nabla\phi(t,x)p^\lambda\left(t+\frac{\theta}{\lambda^2},\d x\right)\right\}\d t Q(\d\theta,\d z)\\
&=\lambda\int_{S}\int_{0}^\infty \hspace{-0.2cm}\alpha(\lambda^2t-\theta,\theta)\left[z+G(\lambda^2t-\theta,\theta)\right]\cdot \left\{\int_{\R^N}\hspace{-0.3cm}\nabla\phi(t,x)p^\lambda(t,\d x)\right\}\d t Q(\d\theta,\d z)\\
&~+O\left(\frac{1}{\lambda}\right).
\end{split}
\end{equation*}
Here the remaining term $O\left(\frac{1}{\lambda}\right)$ depends upon $\|\partial_t\nabla\phi\|_{L^1(0,\infty;W^{1,\infty}(\R^N))}$. As a consequence one obtains that
$$
\mathcal T_1^\lambda=\lambda\int_{\R^+\times\R^N}\mathcal J_1(\lambda^2t)\nabla\phi(t,x)p^\lambda(t,\d x)\d t+O\left(\frac{1}{\lambda}\right),
$$
wherein the function $\mathcal J_1:\R^+\to\R^N$ is defined by
$$
\mathcal J_1(t)=H'(t)+\int_S \alpha(t-\theta,\theta)\left[z+G(t-\theta,\theta)\right] Q(\d\theta,\d z).
$$
Next note that \eqref{asympt} yields
$$
\mathcal J_1(t)=O\left(e^{-\beta t}\right)\text{ as }t\to\infty.
$$
Hence $\mathcal J_1\in L^1(\R^+)$ and Lemma \ref{LE23} applies and ensures that $\mathcal T_1^\lambda=O\left(\frac{1}{\lambda}\right)$, so that
$$
\lim_{\lambda\to\infty}\mathcal T_1^\lambda=0.
$$
\noindent{\bf Step 2:} We are now looking at the second term, namely $\mathcal T_2^\lambda$. During this step we write $\Phi^\lambda$ instead of $\Phi^\lambda[\phi]$. Next let us first notice that one has 
$$
\mathcal T_2^\lambda=\int_{S}\left(-\int_{\frac{\theta}{\lambda^2}}^0+\int_0^\infty\right)\int_{\R^N}\Phi^\lambda\left(t-\frac{\theta}{\lambda^2},\theta,x,z\right)p^\lambda(t,\d x)\d t Q(\d\theta,\d z).
$$
Therefore this yields
$$
\mathcal T_2^\lambda=\int_{S}\int_0^\infty\int_{\R^N}\Phi^\lambda\left(t-\frac{\theta}{\lambda^2},\theta,x,z\right)p^\lambda(t,\d x)\d t Q(\d\theta,\d z)+O\left(\frac{1}{\lambda^2}\right).
$$
Now we claim that:
\begin{claim}\label{Claim26}
The following convergence holds true:
\begin{equation*}
\lim_{\lambda\to\infty} \int_S\Phi^\lambda\left(t-\frac{\theta}{\lambda^2},\theta,x,z\right)Q(\d\theta,\d z)={\frac12}{\rm div}\,\left(D_0\nabla\phi\right),
\end{equation*}
in $L^1\left(0,\infty;\C_0(\R^N)\right)$.
Here the matrix $D_0$ is defined by
\begin{equation}\label{def-D0}
D_0=\int_S \alpha_\infty(\theta)\left[z+\theta {\bf K}\right]\left[z+\theta {\bf K}\right]^T Q(\d\theta,\d z).
\end{equation}
\end{claim}
The proof of this claim is postponed and we first complete the convergence of
$\mathcal T_2^{\lambda_n}$. Here recall that $\{\lambda_n\}_{n\geq 0}$ is the sequence
chosen at the beginning of this proof (see \eqref{29} and \eqref{30}). Now because of the
Young convergence property $p^{\lambda_n}(t,\d x)\to P(t,\d x)$ we obtain that
\begin{equation*}
\begin{split}
&\lim_{n\to\infty}\int_0^\infty\int_{\R^N}\int_S \Phi^{\lambda_n}\left(t-\frac{\theta}{\lambda_n^2},\theta,x,z\right)Q(\d\theta,\d z)p^{\lambda_n}(t,\d x)\d t\\
&={\frac12}\int_{\R^+\times\R^N}{\rm div}\,\left(D_0\nabla\phi\right)P(t,\d x)\d t.
\end{split}
\end{equation*}
This re-writes as 
\begin{equation*}
\lim_{n\to\infty}\mathcal T_2^{\lambda_n}={\frac12}\int_{\R^+\times\R^N}{\rm div}\,\left(D_0\nabla\phi\right)P(t,\d x),
\end{equation*}
that completes the proof of this step.\\
It remains to prove Claim \ref{Claim26}.
To do so, let us observe that there exists some constant $C>0$ such that for all $(t,x)\in \R^+\times \R^N$, $(\theta,z)\in S$, $s\in [0,1]$ and $\lambda\geq 1$ one has, setting $Z=z+G(\lambda^2t-\theta,\theta)$, 
\begin{equation*}
\begin{split}
&\int_S\left|\left[D^2\phi\left(t-\frac{\theta}{\lambda^2}, x+s\frac{Z}{\lambda}\right)-D^2\phi(t,x)\right]\cdot Z^2\right|Q(\d\theta,\d z)\\
&\leq \frac{C}{\lambda}\int_{[-1,0]\times \{|z|\leq \lambda^{1/4}\}}\hspace{-0.5cm}[|z|^3+1]Q(\d\theta,\d z)+C\int_{[-1,0]\times \{|z|\geq \lambda^{1/4}\}}\hspace{-0.5cm}[|z|^2+1]Q(\d\theta,\d z)\\
&\leq \frac{C(1+\lambda^{3/4})}{\lambda}+C\int_{[-1,0]\times \{|z|\geq \lambda^{1/4}\}}[|z|^2+1]Q(\d\theta,\d z)
\end{split}
\end{equation*}
This upper bound converges to zero as $\lambda\to\infty$ so that
\begin{equation*}
\lim_{\lambda\to \infty}\int_S\left|\left[D^2\phi\left(t-\frac{\theta}{\lambda^2}, x+s\frac{Z}{\lambda}\right)-D^2\phi(t,x)\right]\cdot Z^2\right|Q(\d\theta,\d z)=0,
\end{equation*}
uniformly for $(t,x)\in \R^+\times\R^N$ and also in $L^1(0,\infty;\C_b(\R^N))$ since $\phi\in \mathcal S$ and using Lebesgue convergence theorem. 
To complete the proof of the claim, it is sufficient to show that, for the topology of $L^1(0,\infty;\C_b(\R^N))$, one has
\begin{equation*}
\lim_{\lambda\to \infty} \int_ S D^2\phi(t,x)\cdot Z^2 Q(\theta,\d z)=\int_S D^2\phi(t,x)\cdot \left[z+\theta {\bf K}\right]^2Q(\d\theta,\d z).
\end{equation*}
This latter convergence directly follows from the asymptotic expansion of $G(t,\theta)$ in \eqref{asympt}. This complete the proof of Claim \ref{Claim26} by noticing that
\begin{equation*}
\int_S D^2\phi(t,x)\cdot \left[z+\theta {\bf K}\right]^2Q(\d\theta,\d z)={\rm div}\left( D_0\nabla \phi\right),
\end{equation*}
where the matrix $D_0$ is defined in \eqref{def-D0}. This completes the proof of the claim.

\vspace{1ex}

\noindent{\bf Step 3:} In this last step we investigate the limit, as $\lambda\to\infty$, of $\mathcal T^\lambda_3$.
In order to deal with this term, note that
\begin{equation*}
\begin{split}
\mathcal T^\lambda_3=& \lambda^2\int_0^\infty \int_S \alpha(\lambda^2 t,\theta)\int_{\R^N} \phi(t,x)p^\lambda(t,\d x)Q(\d\theta,\d z) \d t\\
&- \lambda^2 \int_S \int_{\theta/\lambda^2}^0 \alpha(\lambda^2 t-\theta)\int_{\R^N} \phi\left(t-\frac{\theta}{\lambda^2},x\right)p^\lambda(t,\d x) \d t Q(\d\theta,\ dz)\\
&-\lambda^2\int_0^\infty \int_S \alpha(\lambda^2 t-\theta,\theta)\int_{\R^N} \phi\left(t-\frac{\theta}{\lambda^2},x\right)p^\lambda(t,\d x)Q(\d\theta,\d z) \d t.
\end{split}
\end{equation*}
For the last term, let us observe that one has
\begin{equation*}
\begin{split}
&\int_0^\infty \int_S \alpha(\lambda^2 t-\theta,\theta)\int_{\R^N} \phi\left(t-\frac{\theta}{\lambda^2},x\right)p^\lambda(t,\d x)Q(\d\theta,\d z) \d t\\
=&\int_0^\infty \int_S \alpha(\lambda^2 t-\theta,\theta)\int_{\R^N} \phi\left(t,x\right)p^\lambda(t,\d x)Q(\d \theta,\d z) \d t\\
&-\int_0^\infty \int_S \frac{\theta}{\lambda^2}\alpha(\lambda^2 t-\theta,\theta)\int_{\R^N} \partial_t \phi\left(t,x\right)p^\lambda(t,\d x)Q(\d \theta,\d z) \d t+O\left(\frac{1}{\lambda^4}\right).
\end{split}
\end{equation*}
This allows us to re-write $\mathcal T^\lambda_3$ as follows:
\begin{equation*}
\begin{split}
\mathcal T^\lambda_3=&\lambda^2 \int_0^\infty\int_{\R^N}F(\lambda^2 t) p^\lambda(t,\d x)\d t\\
&- \lambda^2 \int_S \int_{\theta/\lambda^2}^0 \alpha(\lambda^2 t-\theta)\int_{\R^N} \phi\left(t-\frac{\theta}{\lambda^2},x\right)p^\lambda(t,\d x) \d t Q(\d \theta,\ dz)\\
&+\int_0^\infty \int_S \theta\alpha(\lambda^2 t-\theta,\theta)\int_{\R^N} \partial_t \phi\left(t,x\right)p^\lambda(t,\d x)Q(\d \theta,\d z) \d t+O\left(\frac{1}{\lambda^4}\right).
\end{split}
\end{equation*}
In the above decomposition we have set
\begin{equation*}
F(t)=\int_S \left[\alpha( t,\theta)-\alpha(t-\theta,\theta)\right]Q(\d\theta,\d z).
\end{equation*}
Now, recalling \eqref{asympt}, one has $F(t)=O(e^{-\beta t})$ as $t\to\infty$ so that $F\in L^1(\R^+)$ and Lemma \ref{LE23} applies and ensures that
\begin{equation*}
\lim_{\lambda\to\infty}\lambda^2 \int_0^\infty\int_{\R^N}F(\lambda^2 t) p^\lambda(t,\d x)\d t=\phi(0,0)\int_0^\infty F(s)\d s.
\end{equation*}
Next, using the same argument as for the proof of Lemma \ref{LE23}, one has
\begin{equation*}
\begin{split}
\lim_{\lambda\to\infty}&\lambda^2 \int_S \int_{\theta/\lambda^2}^0 \alpha(\lambda^2 t-\theta)\int_{\R^N} \phi\left(t-\frac{\theta}{\lambda^2},x\right)p^\lambda(t,\d x) \d t Q(\d \theta,\d z)\\
&=\phi(0,0)\int_S \int_{\theta}^0 \alpha(t-\theta)\d t Q(\d \theta,\ dz).
\end{split}
\end{equation*}
And, for the last term we get, along the sequence $\{\lambda_n\}_{n\geq 0}$ (see \eqref{30}),
\begin{equation*}
\begin{split}
\lim_{n\to\infty}&\int_0^\infty \int_S \theta\alpha(\lambda_n^2 t-\theta,\theta)\int_{\R^N} \partial_t \phi\left(t,x\right)p^{\lambda_n}(t,\d x)Q(\d \theta,\d z) \d t\\
&=\int_0^\infty \int_S \theta\alpha_\infty(\theta)Q(\d \theta,\d z)\int_0^\infty\int_{\R^N} \partial_t \phi\left(t,x\right)P(t,\d x)\d t\\
&=-\Gamma \int_0^\infty\int_{\R^N} \partial_t \phi\left(t,x\right)P(t,\d x)\d t.
\end{split}
\end{equation*}
To summarize we have obtained the following property:
\begin{equation*}
\begin{split}
\lim_{n\to\infty}\mathcal T_3^{\lambda_n}=&\phi(0,0)\int_0^\infty F(s)\d s-\phi(0,0)\int_S \int_{\theta}^0 \alpha(t-\theta)\d t Q(\d \theta,\ dz)\\
&-\Gamma \int_0^\infty\int_{\R^N} \partial_t \phi\left(t,x\right)P(t,\d x)\d t.
\end{split}
\end{equation*}
And, to complete this step, note that
\begin{equation*}
\begin{split}
&\int_0^\infty F(s)\d s-\int_S \int_{\theta}^0 \alpha(t-\theta)\d t Q(\d \theta,\ dz)\\
=&\lim_{M\to\infty}\left[\int_S \int_0^M \left[\alpha( t,\theta)-\alpha(t-\theta,\theta)\right]Q(\d\theta,\d z)\int_{\theta}^0 \alpha(t-\theta)\right]\d t\\
=&-\lim_{M\to \infty}\int_S \int_M^{M-\theta}\alpha(t,\theta)\d t Q(\d \theta,\ dz)=\int_S \theta \alpha_\infty(\theta)Q(\d \theta,\ dz)=-\Gamma.
\end{split}
\end{equation*}
Hence we get
\begin{equation*}
\lim_{n\to\infty}\mathcal T_3^{\lambda_n}=-\Gamma \phi(0,0)-\Gamma \int_0^\infty\int_{\R^N} \partial_t \phi\left(t,x\right)P(t,\d x)\d t.
\end{equation*}

We now conclude the proof of Proposition \ref{PROP20} and thus the one of Theorem \ref{THEO19}.\\
\noindent{\bf Conclusion of the proof of Proposition \ref{PROP20}:}\\
From the four previous steps we have obtained the following convergence: for all $\phi\in \mathcal S(\R^+\times\R^N)$, on the one hand, one has
\begin{equation*}
\lim_{n\to \infty}\int_{\R^+\times \R^N}\partial_t \phi(t,x) p^{\lambda_n}(t,\d x)\d t=\int_{\R^+\times\R^N}\partial_t \phi(t,x)P(t,\d x)\d t,
\end{equation*}
and, on the other hand, one has 
\begin{equation*}
\begin{split}
&\lim_{n\to \infty}\int_{\R^+\times \R^N}\partial_t \phi(t,x) p^{\lambda_n}(t,\d x)\d t=-(1+\Gamma)\phi(0,0)\\
&-{\frac12}\int_{\R^+\times\R^N}{\rm div}\,\left(D_0\nabla\phi\right)P(t,\d x)-\Gamma \int_{\R^+\times \R^N} \partial_t \phi\left(t,x\right)P(t,\d x)\d t.
\end{split}
\end{equation*}
Hence $P\in \mathcal S'(\R^+\times\R^N)$ is a solution of the equation
\begin{equation*}
\partial_t \Pi=\delta_{t=0}\otimes \delta_0+{\frac12}{\rm div}\left(\frac{D_0}{1+\Gamma}\nabla \Pi\right)\text{ in }\mathcal S'(\R^+\times\R^N).
\end{equation*}
The latter equation has a unique tempered distribution, $\Pi$, that satisfies $\Pi\in \mathcal C\left(\R^+,\mathcal P\right)$.
Finally, since the limit function $\Pi$ is unique, one obtains $P(t,\d x)=\Pi(t,\d x)$ and this complete the proof of Proposition \ref{PROP20} and thus the one of Theorem \ref{THEO19} since the sequence $(\lambda_n)$ is arbitrary. 

\section{Proof of Theorem \ref{THEO2}}\label{sec:proofTHEO2}

In this section, we complete the proof of Theorem \ref{THEO2}. To that aim we will prove that the family $\{p^\lambda(t,\d x)\}_{\lambda>0}$ is relatively compact with respect to a stronger topology (in time) than those of the Young measures. We shall more specifically show that this family is relatively compact for the topology of $\mathcal C_{\rm loc}\left((0,\infty);H_{\rm loc}^{-\sigma}(\R^N)\right)$ for some parameter $\sigma>0$ large enough. And, this compactness property will be sufficient to complete the proof of the Theorem \ref{THEO2} by using the identification of the weak limit as described in Theorem \ref{THEO19}.  

The main result of this section reads as follows.
\begin{theorem}\label{THEO-compact}
Let $R>0$, $0<\varepsilon<T$ be given. Let $\sigma>\frac{N}{2}+2$ be given. Then there exists $\tilde \lambda>0$ such that the family $\{p^\lambda(t,\d x)\}_{\lambda>\tilde \lambda}$ is relatively compact in $\mathcal C\left([\varepsilon,T];H^{-\sigma-1}(B_R)\right)$. Herein $B_R\subset \R^N$ denotes the ball of radius $R$ centred at the origin while for all $s>\frac{N}{2}+1$, $H^{-s}(B_R)=\left(H_0^s (B_R)\right)'$ denotes the dual space of the Hilbert space $H_0^s(B_R)$.
\end{theorem}

The proof of this key result is postponed and we first complete the
proof of Theorem \ref{THEO2}.\\
\begin{proof}[Proof of Theorem \ref{THEO2}]
Recalling the definition of $p^\lambda(t,\d x)$ given in \eqref{def-p-lambda}, to prove Theorem \ref{THEO2}, it is sufficient to prove that
\begin{equation}\label{pr}
\lim_{t\to\infty} p^{\sqrt t}(1,\d x)=\Pi_1(\d x)\text{ for the narrow topology of $\mathcal P$}.
\end{equation}
To prove the above statement, fix $\sigma>\frac{N}{2}+2$ and consider a sequence $\{\lambda_n\}_{n\geq 0}\subset (0,\infty)$ such that $\lambda_n\to \infty$ as $n\to\infty$. Then because of Theorem \ref{THEO19}, Theorem \ref{THEO-compact} and using a diagonal extraction process, there exists a sub-sequence, still denoted by $\lambda_n$, such that
\begin{equation*}
\lim_{n\to\infty} p^{\lambda_n}(t,\d x)=\Pi_t\text{ locally uniformly for $t\in (0,\infty)$ with value in $H^{-\sigma-1}(B_R)$}.
\end{equation*} 
Note also that, possibly along a sub-sequence, the sequence of probability measure $\{p^{\lambda_n}(1,\d x)\}$ converges to some positive measure with respect to the vague topology of measures. Hence, one obtained that
$$
p^{\lambda_n}(1,\d x)\to \Pi_1(\d x)\text{ as $n\to \infty$},
$$
wherein the above limit holds with respect to the vague topology of measures. Finally, since $\Pi_1(\R^N)=1$, $\Pi_1$ is also  a probability measure and the above limit holds for the narrow topology of probability measures.
Now, since the sequence $\{\lambda_n\}$ is arbitrary and $\mathcal P$, endowed with the narrow topology, is a metrizable space, one obtains that
\begin{equation*}
\lim_{\lambda\to\infty} p^{\lambda}(1,\d x)=\Pi_1\text{ for the narrow topology of $\mathcal P$}.
\end{equation*}
Thus \eqref{pr} holds true by choosing $\lambda=\sqrt{t}$. This completes the proof of Theorem \ref{THEO2}.

\end{proof}

In the remaining parts of this section, we prove Theorem \ref{THEO-compact}. The proof of this result relies on the formulation of the equation for $p^\lambda$ obtained in Lemma \ref{LE24}. Here we re-write it with a slightly different form, as follows: for $\phi\in W^{1,1}\left(0,\infty;\C_b(\R^N)\right)\cap L^1(0,\infty;\C_b^2(\R^N))$ and $\lambda>0$, the function $t\mapsto p^\lambda(t,\d x)$ satisfies
\begin{equation}\label{310}
\begin{split}
\int_{\R^+\times\R^N} \left[T_{\frac{1}{\lambda^2}}\partial_t \phi\right](t,x)p^\lambda(t,\d x)\d t=&
-\int_{\R^N}\phi(0,x)p^\lambda(0,\d x)-\mathcal{T}_1^\lambda[\phi]\\
&-\mathcal T_2^\lambda[\phi]+\mathcal R^\lambda[\phi],
\end{split}
\end{equation}
wherein $\mathcal T_1^\lambda$ and $\mathcal T_2^\lambda$ are defined in Lemma \ref{LE24}, $T_{\frac{1}{\lambda^2}}$ denotes the nonlocal operator defined by
\begin{align*}
~&g(t,x)=[T_{\frac{1}{\lambda^2}}f](t,x)\\
\Leftrightarrow\;&g(t,x)=f(t,x)+\lambda^2\int_{S}\alpha_\infty(\theta)\int_0^{-\frac{\theta}{\lambda^2}}f(t+l,x)\d l Q(\d\theta,\d z),
\end{align*}
while
\begin{equation*}
\begin{split}
\mathcal{R}^\lambda[\phi]&=\lambda^2\int_0^\infty \int_{S}\left[\alpha(\lambda^2t,\theta)-\alpha_\infty(\theta)\right]\\
&\qquad\qquad\qquad\int_{\R^N}\phi(t,x)\left[p^\lambda(t,\d x)-p^\lambda\left(t+\frac{\theta}{\lambda^2},\d x\right)\right] Q(\d \theta,\d z) \d t\\
&~+\lambda^2\int_S \alpha_\infty(\theta)\int_{\frac{\theta}{\lambda^2}}^0\int_{\R^N}\phi\left(t-\frac{\theta}{\lambda^2},x\right)p^\lambda(t,\d x) \d t Q(\d \theta,\d z).
\end{split}
\end{equation*}
Now, in order to prove Theorem \ref{THEO-compact} we will investigate some regularization properties for the nonlocal operator $T_{\frac{1}{\lambda^2}}$ defined above when $\lambda>\!\!>1$ is large. And, we will be used to complete the proof of the theorem. In the sequel we will first investigate some properties of the linear operator $T_{\frac{1}{\lambda^2}}$ before going to the proof of Theorem \ref{THEO-compact}.

\subsection{Regularity properties}

Let $\left(X,\|\cdot\|_X\right)$ be a reflexive Banach space. For each $h>0$, consider the linear operator $T_h$ defined from $\mathcal S\left(\R^+;X'\right)$ into itself by
\begin{equation}\label{Th}
T_h[\phi](t)=\phi(t)+\frac{1}{h}\int_S\alpha_\infty(\theta)\int_0^{-h\theta}\phi(t+l)\d l Q(\d\theta,\d z),\;t\geq 0.
\end{equation}
Here $X'$ denotes the dual space of $X$. Then the main result of this sub-section reads a s follows:
\begin{theorem}[Sobolev semi-norm estimate]\label{THEO-reg} Let $\{u^h\}_{h>0}$ be a family of $X$-valued tempered distribution on $\R^+$, namely $\{u_h\}_{h>0}\subset \mathcal S'\left(\R^+;X\right)$. We assume that there exist some constant $M>0$ and $\alpha\in (0,1)$ such that for all $h>0$ and all $\phi\in \mathcal S\left(\R^+;X'\right)$ one has:
\begin{align}\label{eq-35}
&\left|\left\langle T_h\phi',u^h\right\rangle\right|\nonumber\\
&\leq M\left[\|\phi\|_{L^\infty(\R^+;X')}+\|\phi\|_{L^1(\R^+;X')}+\sup_{s>0}\frac{1}{s^\alpha}\|\phi(\cdot+s)-\phi\|_{L^1(\R^+;X')}\right].
\end{align}
Herein the symbol $\langle\cdot,\cdot\rangle$ is used to denote the duality pairing between $\mathcal S(\R^+;X')$ and $\mathcal S'(\R^+;X)$. Then, for each $p>1$ and $s\in [\alpha,1)$ with $ps>1$, one has
$$
u^h\in W^{1-s,p'}_{\rm loc}\left(\R^+;X\right),\;\forall h>0,
$$
where $p'$ denotes the conjugate exponent of $p$.
Moreover for each given $T>0$, there exists some constant $C=C(p,s,T)$ such that
$$
\left[u^h\right]_{W^{1-s,p'}(0,T;X)}\leq C,\;\forall h>0.
$$  
In the above estimate, the bracket denotes the $X-$valued Sobolev semi-norm described below.
\end{theorem}

The proof of this result relies on several preliminary studies.
As announced in Theorem \ref{THEO-reg}, the estimates we shall obtain make use of Banach
valued fractional Sobolev spaces. For that purpose let us first introduce further
notations, definitions and some well known results and estimates related to
Banach valued fractional Sobolev spaces. We refer to \cite{1, 19} and the references
therein for more details on such a topic.

Let $(Y,\|\cdot\|_Y\|)$ be a given Banach space. For each open interval $I\subset \R$, each
$p\in [1,\infty]$ and $\theta\in (0,1)$, we define the Sobolev space $W^{\theta,p}(I;Y)$ as follows:
\begin{equation*}
W^{\theta,p}(I;Y)=\{\psi\in L^p(I;Y):\;[\psi]_{W^{\theta,p}(I;Y)}<\infty\},
\end{equation*}
where the Sobolev semi-norm $[\cdot]_{W^{\theta,p}(I;Y)}$ is defined by
\begin{equation}\label{sob}
[\psi]_{W^{\theta,p}(I;Y)}=\left(\iint_{I\times I}\frac{\|\psi(t)-\psi(s)\|_Y^p}{|t-s|^{\theta p+1}}\d t\d s\right)^{\frac{1}{p}}.
\end{equation}
The Sobolev norm is then defined as
$$
\|\psi\|_{W^{\theta,p}(I;Y)}=\|\psi\|_{L^{p}(I;Y)}+[\psi]_{W^{\theta,p}(I;Y)},\;\forall \psi\in W^{\theta,p}(I;Y).
$$
Now we define Sobolev spaces with negative index. When $I\subset\R$ is a given
bounded interval, then for $p\in [1,\infty]$ and $\theta\in (0,1)$  we define the Sobolev
space $W^{-1+\theta,p}(I; Y )$ as the image of the distributional derivative $\partial_t$, that is
$\partial_t W^{\theta,p}(I, Y ) \subset \mathcal D' (I; Y )$, the set of $Y-$valued distributions. In other word it is
defined as the following:
$$
u\in W^{-1+\theta,p}(I; Y )\;\Leftrightarrow\;\exists v\in W^{\theta,p}(I; Y ),\;\;u=\partial_t v\text{ in }\mathcal D' (I; Y ).
$$
Such a function $v$ is called a representation of $u$. The set of the representations of $u$
is given by $v + Y$ where $v$ is a representation of $u$. The space $W^{ −1+\theta,p} (I; Y )$ becomes a
Banach space when endowed with the norm $\|\cdot\|_{W^{ −1+\theta,p} (I;Y )}$ defined as
$$
\|u\|_{W^{ −1+\theta,p} (I;Y )}=\inf_{y\in Y}\|v+y\|_{L^p(I;Y)}+[v]_{W^{\theta,p} (I;Y )},
$$
where $v$ is a given representation of $u$. Note that such a definition does not
depend upon the choice of the representation $v$.

Now let us recall a duality representation that will be used to prove Theorem
\ref{THEO-reg}. The proof of this result can be found in \cite{1}.

\begin{lemma}
Let $I\subset\R$ be a given interval. Let $(Y,\|\cdot\|_Y )$ be a given separable
Banach space. Let $p\in (1,\infty)$ and $\theta\in (0,1)$ be given. Then one has
$$
W^{-\theta,p'} (I; Y ) = \left(W_0^{\theta,p} (I; Y' )\right)'.
$$
Here $p'$ denotes the conjugate exponent associated to $p$, $Y'$ is the dual space of $Y$ while
$W_0^{\theta,p} (I; Y' )$ is the closure of $\mathcal D\left(I; Y'\right )$ in $W^{\theta,p} \left(I; Y' \right)$. The above duality representation holds with respect to the duality pairing $\langle\,\cdot\,;\,\cdot\,\rangle$ on $\mathcal D(I; Y ) \times\mathcal D(I; Y' )$
defined by
$$
\langle v;v'\rangle=\int_I\langle v(t);v'(t)\rangle_{Y,Y'} \d t,\;\forall (v,v')\in \mathcal D(I; Y ) \times\mathcal D(I; Y' ).
$$
\end{lemma}

Before recalling some important estimates that will be used in the sequel, let us
introduce further notations.
Let $I\subset\R$ be a given interval and let $(Y,\|\cdot\|_Y )$ be a Banach space. For each
$p\in [1,\infty]$ and $s\in (0,1)$ let us define the Besov semi-norm for a function
$\phi\in L^p_{\rm loc} (I; Y )$ by
\begin{equation*}
[\phi]_{B_\infty^{s,p}(I;Y)}=\sup_{h>0} h^{-s}\|\phi(\cdot+h)-\phi\|_{L^p(I_h;Y)}\text{ with }I_h=\{x\in I:\;x+h\in I\}.
\end{equation*}

We now turn to derive some important estimates.
Using straightforward computations we can derive the following estimates:
\begin{lemma}\label{LE29} Let $\theta\in (0, 1)$ be given. Let $T > 0$ be given. Then for all $\psi\in
\mathcal D (0, T ; Y )$ one has
$$
[\bar \psi]_{ W^{\theta,1} (\R;Y )}\leq  [\psi]_{ W^{\theta,1} (0,T ;Y )} +
\frac{4T^{1-\theta}}{\theta(1-\theta)}\|\psi\|_{L^\infty (0,T ;Y )}.
$$
Here $\bar\psi$ denotes the extension by zero of $\psi$.
\end{lemma}
Next let us recall the estimate derived by Simon in Corollary 24 of \cite{19}:
\begin{lemma}\label{LE30}
Let $I\subset \R$ be a given interval and let $(Y,\|\cdot\|_Y )$ be a Banach space.
Let $\theta\in  (0, 1)$ and $p \in [1, \infty)$ be given. Then for each $\psi\in W^{\theta,p} (I, Y )$ one has
\begin{equation*}
[\psi]_{B_\infty^{\theta,p}(I;Y)}\leq \frac{2}{\theta}[\psi]_{W^{\theta,p}(I;Y)}.
\end{equation*}

\end{lemma}

Now in view of proving Theorem \ref{THEO-reg}, we are concerned in deriving some
properties of the linear operator $T_h$ defined above in \eqref{Th}. Roughly speaking,
we shall show that such an operator is invertible on $L^1 (\R; X' )$ and that the
inverse operator is a kernel operator that enjoys nice estimates.
To that aim, we first investigate some properties of the function $\Delta:\mathbb C\to \mathbb C$ defined by
\begin{equation*}
\Delta (s)=s-\int_S\alpha_\infty(\theta)[e^{s\theta}-1]\;Q(\d\theta,\d z).
\end{equation*}
The following lemma holds true:
\begin{lemma}\label{LE-carac}
The analytic function $\Delta$ enjoys the following property: there exists $\varepsilon_0>0$ such that
\begin{equation*}
\begin{cases}
\Re(s)>-\varepsilon_0,\\
\Delta(s)=0,
\end{cases}\Rightarrow\;s=0.
\end{equation*}
\end{lemma}

\begin{proof}
Note that $\Delta(0)=0$. Next let $s=\alpha+i\omega$ with $\alpha\geq 0$ and $\omega$ be such that $\Delta(s)=0$. Let us show that $\alpha=\omega=0$.\\
To do so, let us first observe that the function $x\mapsto \Delta(x)$ is increasing on $\R$. Hence $s=0$ is the only real root of $\Delta$.
Next the equation $\Delta(\alpha+i\omega)=0$ re-writes as
\begin{equation*}
\begin{cases}
\alpha+\int_S\alpha_\infty(\theta)Q(\d\theta,\d z)=\int_S\alpha_\infty(\theta)e^{\alpha\theta}\cos(\omega\theta)\;Q(\d\theta,\d z),\\
\omega-\int_S\alpha_\infty(\theta)e^{\alpha\theta}\sin(\omega\theta)\;Q(\d\theta,\d z)=0.
\end{cases}
\end{equation*}
Hence, from the first equation one gets $\Delta(\alpha)\leq 0$, that implies that $\alpha\leq 0$ and thus $\alpha=0$ since $\alpha\geq 0$. Next the first equation becomes
$$
\int_S\alpha_\infty(\theta)[1-\cos(\omega\theta)]Q(\d \theta,\d z)=0,
$$
and, since $\alpha_\infty>0$, this implies that $\cos(\omega\theta)=1$ $Q-a.e.$ for $(\theta,z)\in S$ and thus $\sin(\omega\theta)=0$ $Q-a.e.$ for $(\theta,z)\in S$. Plugging this information together with $\alpha=0$ into the second equation in the above system yields $\omega=0$. The above argument shows hat
$$
\Re(s)\geq 0\text{ and }\Delta(s)=0\;\Rightarrow\;s=0.
$$
Recalling that $\Delta$ is an analytic function, to complete the proof of the lemma, it is sufficient to note that if $s\in\mathbb C$ satisfies $\Delta(s)=0$ then
\begin{equation*}
|\Im(s)|\leq \int_S \alpha_\infty(\theta)e^{\theta\Re(s)}Q(\d\theta,\d z).
\end{equation*} 
Indeed the above estimate ensures that there is no sequence of roots for $\Delta$ approaching the imaginary axis. This completes the proof of the lemma.
\end{proof}

Observe furthermore that $\Delta'(0)=1-\int_S \theta\alpha_\infty(\theta) Q(\d\theta,\d z)>0$ so that $s=0$ is a simple root and the function $s\mapsto \frac{s}{\Delta(s)}$ is holomorphic on the half plane $\{s\in\mathbb C:\;\Re(s)>-\varepsilon_0\}$.

\begin{lemma} \label{LE-repre} For each function $\psi\in L^1 (\R, X' )$, for each parameter $h > 0$, there
exists a unique function $\phi=\phi_h\in  L^1 (\R; X' )$ such that
\begin{equation*}
\psi = T_h [\phi].
\end{equation*}
Moreover there exists a real valued function $K \in L^1 (\R; \R)$ such that there exist
$\varepsilon_0 > 0$ small enough and some constant $M_0 > 0$ such that:
\begin{equation}\label{esa}
|K(x)| \leq M_0 e^{ -\varepsilon_0 |x|} , \;\forall x \in \R,\;\;
K(x) = 0 \text{ a.e. } x \leq 0,
\end{equation}
and such that for all $\psi\in L^1 (\R, X' )$, for all $h > 0$, one has
\begin{equation}\label{form}
\psi=T_h[\phi]\;\Leftrightarrow\;\phi(t)=\psi(t)+\frac{1}{h}\int_\R K\left(\frac{t-s}{h}\right)\psi(s)\d s,\;t\in\R.
\end{equation}

\end{lemma}
In the sequel we shall denote for each $h > 0$:
$$
K_h (x) =\frac{1}{h}K\left(h^{-1}x\right).
$$
Note that the above bound, namely \eqref{esa}, for the convolution kernel $K_h$ ensures
that
$$
\|K_h\|_{L^1(\R)}\leq M_0\varepsilon_0^{-1},\;\forall h>0.
$$
\begin{proof}
Let $h > 0$ be given. Let $\psi\in L^1 (\R; X' )$ be given. We aim at solving
the equation $T_h [\phi] = \psi$. We shall denote by $\mathcal F$ the Fourier transform. Then
applying the Fourier transform yields
\begin{equation*}
\begin{split}
\mathcal F T_h[\phi](\xi)=&\left[1+\frac{1}{h}\int_S\alpha_\infty(\theta)\int_0^{-h\theta} e^{-il\xi}\d l\;Q(\d\theta,\d z)\right]\mathcal F\phi(\xi)=\frac{\Delta(ih\xi)}{ih\xi}\mathcal F\phi(\xi).
\end{split}
\end{equation*}
Now we claim that:
\begin{claim}\label{claim-L1}
There exists a function $K:\R\to\R$ and $\varepsilon>0$ such that
\begin{equation*}
K(x)=O(e^{-\varepsilon x}),\;x\to\infty,\;\;K(x)=0\text{ for $x\leq 0$},
\end{equation*}
and
\begin{equation*}
\frac{i\xi}{\Delta(i\xi)}=1+\mathcal F K(\xi),\;\forall \xi\in\R.
\end{equation*}
\end{claim}
Before proving this claim we first complete the proof of Lemma \ref{LE-repre}. Indeed, due to the above claim, one has
\begin{equation*}
\frac{ih\xi}{\Delta(ih\xi)}=1+\mathcal F K_h(\xi),\;\forall \xi\in\R,\;h>0.
\end{equation*}
Hence the function $\phi\in L^1(\R; X')$ is uniquely defined by \eqref{form} and Lemma \ref{LE-repre} follows.
\end{proof}

It remains to prove Claim \ref{claim-L1}.\\
\begin{proof}[Proof of Claim \ref{claim-L1}]
Recall that the function $\Delta:\mathbb C\to \mathbb C$ is defined by
$$
\Delta(s)=s-\int_S\alpha_\infty(\theta)[e^{s\theta}-1]\;Q(\d\theta,\d z), \;\forall s\in\mathbb C,
$$
and that there exists $\varepsilon>0$ such that $s\mapsto \frac{s}{\Delta(s)}$ is analytic on the half plane $\mathbb H:=\{s\in\mathbb C:\;\Re(s)>-\varepsilon\}$.\\
Next setting 
$$\gamma:=\int_S\alpha_\infty(\theta)Q(\d\theta,\d z) \text{ and }I(s)=\int_S\alpha_\infty(\theta)e^{s\theta}Q(\d\theta,\d z),$$
one has
\begin{equation*}
\begin{split}
G(s):=&\frac{s}{\Delta(s)}=\frac{s+I(s)}{s+\gamma}+\frac{s}{s+\gamma-I(s)}-\frac{s+I(s)}{(s+\gamma)}\\
=&\frac{s+I(s)}{s+\gamma}-\frac{I(s)(\gamma-I(s))}{\Delta(s)(s+\gamma)},\;\forall s\in \mathbb H.
\end{split}
\end{equation*}
Hence one obtains the following decomposition, for any $s\in\mathbb H$,
\begin{equation*}
G(s)=1-\frac{\gamma}{s+\gamma}+\frac{I(s)}{s+\gamma}+R(s)\text{ with }R(s):=\frac{I(s)(I(s)-\gamma)}{\Delta(s)(s+\gamma)}.
\end{equation*}
Up to reducing $\varepsilon>0$ if necessary, one may suppose that $0<\varepsilon<\gamma$ so that each term arising in the above decomposition is analytic of the half plane $\mathbb H$.

Next consider the function $E=E(x)$ defined by
\begin{equation}\label{function-E}
E(x)=\begin{cases} 0 &\text{ if }x\leq0,\\ e^{-x} &\text{ if $x>0$}.\end{cases}
\end{equation}
And first observe that, for all $\xi\in\R$, one has
\begin{equation*}
\frac{\gamma}{i\xi+\gamma}=\mathcal F\left[\gamma E(\gamma\cdot)\right](\xi).
\end{equation*}
Observe that one also has, for any $\xi\in\R$,
\begin{equation*}
\frac{I(i\xi)}{i\xi+\gamma}=\mathcal F\left[\int_S E(\gamma\cdot+\theta)Q(\d\theta,\d z)\right](\xi).
\end{equation*}
On the other hand, observe that the function $I=I(s)$ is bounded on the complex half plane $\mathbb H$. Indeed, one has 
\begin{equation*}
|I(s)|\leq M:=\int_S\alpha_\theta(\theta)e^{-\varepsilon\theta}Q(\d\theta,\d z),\;\forall s\in\mathbb H.
\end{equation*}
As a consequence, there exists some constant $C>0$ such that
\begin{equation*}\label{bound23}
|R(s)|\leq \frac{C}{1+|s|^2},\;\forall s\in\mathbb H.
\end{equation*}
As a consequence of this fast decay at infinity and since $R$ is analytic on $\mathbb H$, one
obtains, for all $x \in \R$ and all $\kappa\geq -\varepsilon$, that
\begin{equation*}
K_R (x) := \mathcal F^{-1} R(i\cdot)(x) =\frac{1}{2\pi}\int_{-\infty}^\infty e^{ix\xi} R(i\xi)\d \xi=\frac{1}{2\pi}\int_{-\infty}^\infty e^{\kappa x+ix\xi}R(\kappa+i\xi)\d \xi.
\end{equation*}
Hence choosing $\kappa=-\varepsilon$, we conclude that the function $K_R$ has an exponential
decay at infinity. In addition, because of the uniform bound in \eqref{bound23}, Theorem
4.4 in \cite{20} ensures that there exists a tempered distribution $T$ supported in
$[0,\infty)$ such that for all $\kappa>0$ and $\xi\in\R$ one has
\begin{equation*}
R \left(i \left(\xi+i(\kappa-\varepsilon)\right)\right) = \mathcal F \left(e^{-\kappa\cdot}T\right)(\xi).
\end{equation*}
Choosing $\kappa=\varepsilon>0$ yields $K_R = e^{-\varepsilon\cdot}T$ and one concludes that ${\rm supp}\, K_R \subset [0,\infty)$.

As a consequence of the above steps, one obtained that for each $\xi\in\R$,
\begin{equation*}
G\left(i\xi\right)=1+\mathcal F\left[-\gamma E(\gamma\cdot)+\int_S E(\gamma\cdot+\theta)Q(\d\theta,\d z)+K_R\right](\xi).
\end{equation*}
Finally, recalling the definition of $E$ in \eqref{function-E}, ${\rm supp}\, K_R\subset [0,\infty)$ and the exponential decay of $K_R$, this completes the proof of Claim \ref{claim-L1}. 

\end{proof}

\begin{remark}\label{REM-decay} Note that if $\psi\in \mathcal D(\R;X')$ then, for each $h>0$, the solution $\phi=T_h^{-1}\psi$ satisfies $\phi\in \mathcal C^\infty(\R;X')$. Moreover one has ${\rm supp}\,\phi\subset {\rm supp}\,\psi+[0,\infty)$ and, for each $k\geq 0$, $\phi^{(k)}(x)=O(e^{-\varepsilon_0 x/h})$ as $x\to\infty$. Hence $\phi$ belongs to the Schwartz class, namely $T_h^{-1}\mathcal D(\R;X')\subset \mathcal S(\R;X')$.  
\end{remark}

Before going to the proof of Theorem \ref{THEO-reg}, we need to derive further regularity estimates described below.

\begin{lemma}\label{LE-reg+}
Let $T>0$ be given. Then, for each $h>0$ and each $\psi\in \mathcal D(0,T;X')$, the function $\phi_h:=T_h^{-1}\psi$ satisfies the following estimates:
\begin{itemize}
\item[(i)] $\phi_h=0$ on $(-\infty,0]$, $\phi_h\in L^\infty(0,\infty;X')$ and
\begin{equation*}
\|\phi_h\|_{L^\infty(0,\infty;X')}\leq \left(1+\frac{M_0}{\varepsilon_0}\right)\|\psi\|_{L^\infty(0,T;X')};
\end{equation*}
\item[(ii)] for each $p\in [1,\infty)$ $\phi_h\in L^p(0,\infty;X')$ and 
\begin{align*}
&\|\phi_h\|_{L^p(0,\infty;X')}\\
&\leq \left(1+\frac{M_0}{\varepsilon_0}\right)\|\psi\|_{L^p(0,T;X')}\leq T^\frac{1}{p}\left(1+\frac{M_0}{\varepsilon_0}\right)\|\psi\|_{L^\infty(0,T;X')};
\end{align*}
\item[(iii)] for each $\alpha\in (0,1)$ and $p\in [1,\infty)$ one has
$$
\left[\phi_h\right]_{B^{\alpha,p}_\infty(0,\infty;X')}\leq \left(1+\frac{M_0}{\varepsilon_0}\right)[\psi]_{B^{\alpha,p}_\infty(\R,X')}.
$$
\end{itemize}
\end{lemma}

\begin{proof}
Recall that $\phi_h$ is defined by
$$
\phi_h=T_h^{-1}\psi=\psi+K_h\ast \psi,
$$
wherein $\ast$ denotes the convolution product.
Next one has
\begin{equation*}
\|\phi_h\|_{L^\infty(0,\infty;X')}\leq \left(1+\|K_h|_{L^1(\R)}\right)\|\psi\|_{L^\infty(0,T;X')}\leq \left(1+\frac{M_0}{\varepsilon_0}\right)\|\psi\|_{L^\infty(0,T;X')}.
\end{equation*}
This completes the proof of $(ii)$.\\
Next, for $p\in [1,\infty)$, one also has
\begin{align*}
&\|\phi_h\|_{L^p(0,\infty;X')}\\
&\leq \left(1+\|K_h|_{L^1(\R)}\right)\|\psi\|_{L^p(0,T;X')}\leq T^\frac{1}{p}\left(1+\frac{M_0}{\varepsilon_0}\right)\|\psi\|_{L^\infty(0,T;X')},
\end{align*}
that proves $(ii)$.\\
Finally observe that for each $s>0$ one has
\begin{equation*}
\begin{split}
&\|\phi_h(\cdot+s)-\phi_h\|_{L^p(0,\infty;X')}\\
&\qquad\qquad\qquad\leq \|\psi(\cdot+s)-\psi\|_{L^p(\R;X')}+\|K_h\ast(\psi(\cdot+s)-\psi)\|_{L^p(\R;X')}\\
&\qquad\qquad\qquad\leq \left(1+\frac{M_0}{\varepsilon_0}\right)\|\psi(\cdot+s)-\psi\|_{L^p(\R;X')},
\end{split}
\end{equation*}
and this completes the proof of $(iii)$.
\end{proof}

We are now in position to prove Theorem \ref{THEO-reg}.\\
\begin{proof}[Proof of Theorem \ref{THEO-reg}]
Let $T>0$ and $\psi\in \mathcal D(0,T;X')$ be given. Then for each $h>0$ set $\phi_h:=T_h^{-1}\psi\in \mathcal S(\R;X')$ (see Remark \ref{REM-decay}).

One the one hand, we infer from Lemma \ref{LE29}, \ref{LE30} and \ref{LE-reg+} that, for each $h>0$, one has
$$
[\phi_h]_{B^{\alpha,1}_\infty(0,\infty;X')}\leq \frac{2}{\alpha}\left(1+\frac{M_0}{\varepsilon_0}\right)\left[[\psi]_{W^{\alpha,1}(0,T;X')}+\frac{4T^{1-\alpha}}{\alpha(1-\alpha)}\|\psi\|_{L^\infty(0,T;X')}\right].
$$
Next, recalling that $\phi_h\in \mathcal S(\R,X')$ and $T_h\partial_t=\partial_t T_h$, applying $\phi=\phi_h$ into \eqref{eq-35} ensues that there exists some constant $\widehat M>0$ (depending upon $T$ but independent from $\psi$ and $h$) such that for all $h>0$ and any $\psi\in \mathcal D(0,T;X')$ one has
$$
\left|\langle \partial_t\psi,u^h\rangle_{\mathcal S,\mathcal S'}\right|\leq \widehat M\left[\|\psi\|_{L^\infty(0,T;X')}+\|\psi\|_{W^{\alpha,1}(0,T;X')}\right].
$$
Now let $s \in [\alpha, 1)$ and $p \in (1, \infty)$ be given such that $ps > 1$. Then let us recall
that (see for instance Simon in \cite{19}) that the following continuous embedding
holds true:
$$
W^{s,p}(0,T;X')\hookrightarrow L^\infty(0,T;X')\cap W^{\alpha,1}(0,T;X').
$$
Hence there exists some positive constant, still denoted by $\widehat M$, depending on $T$ such that for all $h>0$ and $\psi\in\mathcal D(0,T;X')$ one has
$$
\left|\langle \partial_t\psi,u^h\rangle_{\mathcal S,\mathcal S'}\right|\leq \widehat M\|\psi\|_{W^{s,p}(0,T;X')}.
$$
As a consequence $\partial_t u^h$, the distribution derivative of $u^h$, satisfies 
$$
\partial_t u^h\in \left(W_0^{s,p}(0,T;X')\right)'=W^{-s,p'}(0,T;X)
$$ and
\begin{equation*}
\|\partial_t u^h\|_{W^{-s,p}(0,T;X)}\leq \widehat M,\;\forall h>0.
\end{equation*}
Hence, for each $T>0$ there exists some constant $M=M(T)$ such that for all $h>0$ one has
\begin{equation*}
u^h\in W^{1-s,p'}(0,T;X)\text{ and }\left[u^h\right]_{W^{1-s,p'}(0,T;X)}\leq M(T).
\end{equation*}
This completes the proof of Theorem \ref{THEO-reg}.
\end{proof}

\subsection{Proof of Theorem \ref{THEO-compact}}

In this section we complete the proof of Theorem \ref{THEO-compact}. To that aim we apply the abstract result derived above in Theorem \ref{THEO-reg}.\\
Our proof will be split into two steps. In the first step we derive a fractional Sobolev regularity estimate of
$\{p^\lambda\}_{\lambda>0}$ using Theorem \ref{THEO-reg}. In the second step we bootstrap this estimates
to conclude the proof of Theorem \ref{THEO-compact}.

In a first step we derive the following lemma.
\begin{lemma}\label{LE-sob}
For each $\sigma>2+\frac{N}{2}$, each $s\in \left[\frac{1}{2},1\right)$ and each $p>\frac{1}{s}$, the family of maps $\{t\mapsto p^\lambda(t,\d x)\}_{\lambda\geq 1}$ is bounded in $W_{\rm loc}^{1-s,p'}\left(0,\infty;H^{-\sigma}(\R^N)\right)$. This means that for each $T>0$ there exists some constant $M=M(\sigma,T,s,p)>0$ such that
\begin{equation*}
\left\|p^\lambda\right\|_{W^{1-s,p'};H^{-\sigma}(\R^N))}\leq M,\;\forall \lambda\geq 1.
\end{equation*} 

\end{lemma}

\begin{proof}
Fix $\sigma>2+\frac{N}{2}$ so that the following continuous embedding holds true:
\begin{equation}\label{embedding}
H^\sigma(\R^N)\hookrightarrow W^{2,\infty}(\R^N).
\end{equation}
Now observe that the above continuous embedding ensures that for each $\lambda>0$ and each $\phi\in H^\sigma(\R^N)$, the maps $t\mapsto \int_{\R^N}\phi(x) p^\lambda(t,\d x)$ is measurable. Hence, since $H^{-\sigma}(\R^N)=\left(H^\sigma(\R^N)\right)'$ is separable then the map $t\mapsto p^\lambda(t,\d x)$ is strongly measurable from $[0,\infty)$ into $H^{-\sigma}(\R^N)$.
Moreover, there exists some constant $C=C(\sigma)>0$ such that for any $t\geq 0$, $\lambda>0$ and $\phi\in H^\sigma(\R^N)$:
\begin{equation*}
\left|\int_{\R^N}\phi(x)p^\lambda(t,\d x)\right|\leq \|\phi\|_{L^\infty(\R^N)}\leq C(\sigma)\|\phi\|_{H^\sigma(\R^N)}.
\end{equation*}
Hence, one first conclude that the family of measures $\{p^\lambda\}_{\lambda>0}$ is bounded in $L^\infty\left(0,\infty;H^{-\sigma}(\R^N)\right)$, namely
\begin{equation}\label{Linfty}
\left\|p^\lambda\right\|_{L^\infty(0,\infty;H^{-\sigma}(\R^N)}\leq  C(\sigma),\;\forall \lambda>0.
\end{equation}
To go further, recall that $p^\lambda$ satisfies \eqref{310}. Hence we will apply Theorem \ref{THEO-reg} to this formulation. To that aim we estimate the different terms arising in the right hand side of \eqref{310}.

\vspace{1ex}
\noindent {\bf \underline{Estimate for $\mathcal T_1^\lambda$}:}\\ 
Recall that $\mathcal T_1^\lambda$ is defined in Lemma \ref{LE24}. 
Note that, due to \eqref{asympt}, $\mathcal T_1^\lambda$ re-writes as follows, for all $\phi\in \mathcal S(\R^+;H^\sigma(\R^N))$,
\begin{equation*}
\begin{split}
&\mathcal{T}_1^\lambda[\phi]=I_1^\lambda[\phi]-\lambda\int_{0}^\infty \int_{\R^N}\nabla \phi(t,x){\bf K}p^\lambda(t,\d x)\d t\\
&+\lambda \int_0^\infty \int_S \alpha_\infty(\theta)\left[z+\theta {\bf K}\right]\left[\int_{\R^N}\nabla \phi(t,x)p^\lambda\left(t+\frac{\theta}{\lambda^2},\d x\right)\right]Q(\d\theta,\d z)\d t.
\end{split}
\end{equation*}
Herein $I_1^\lambda[\phi]$ is a remainder that satisfies that there exists some constant $M>0$ such that for all $\lambda>0$ and $\phi\in \mathcal S(\R^+;H^\sigma(\R^N))$ one has
$$
\left|I_1^\lambda[\phi]\right|\leq \frac{M}{\lambda}\|\phi\|_{L^1(0,\infty;W^{1,\infty}(\R^N))}\leq \frac{M}{\lambda}\|\phi\|_{L^1(0,\infty;H^\sigma(\R^N))}.
$$
Next we re-write $\mathcal T_1^\lambda$ as follows:
\begin{equation*}
\begin{split}
&\mathcal{T}_1^\lambda[\phi]=I^\lambda[\phi]-\lambda\int_{0}^\infty \int_{\R^N}\nabla \phi(t,x){\bf K}p^\lambda(t,\d x)\d t\\
&+\lambda \int_0^\infty \int_S \alpha_\infty(\theta)\left[z+\theta {\bf K}\right]\left[\int_{\R^N}\nabla \phi(s-\frac{\theta}{\lambda^2},x)p^\lambda\left(s,\d x\right)\right]Q(\d\theta,\d z)\d s\\
&+I_2^\lambda[\phi]
\end{split}
\end{equation*}
with 
$$
I_2^\lambda[\phi]=\lambda \int_S \int_{0}^{-\frac{\theta}{\lambda^2}}\alpha_\infty(\theta)\left[z+\theta {\bf K}\right]\left[\int_{\R^N}\nabla \phi(t,x)p^\lambda\left(t+\frac{\theta}{\lambda^2},\d x\right)\right]\d t Q(\d\theta,\d z).
$$
And note that, as for $I_1^\lambda$, there exists $M>0$
such that for all $\lambda>0$ and $\phi\in \mathcal S(\R^+;H^\sigma(\R^N))$ one has
$$
\left|I_2^\lambda[\phi]\right|\leq \frac{M}{\lambda}\|\phi\|_{L^1(0,\infty;W^{1,\infty}(\R^N))}\leq \frac{M}{\lambda}\|\phi\|_{L^1(0,\infty;H^\sigma(\R^N))}.
$$
Next recalling the identity for ${\bf K}$ in Remark \ref{REM-identity}, $\mathcal T_1^\lambda$ re-writes, for any $\phi\in \mathcal S(\R^+;H^\sigma(\R^N))$ as
\begin{equation*}
\begin{split}
&\mathcal{T}_1^\lambda[\phi]=I^\lambda[\phi]+I_2^\lambda[\phi]+I_3^\lambda[\phi]\text{ with }\\
&I_3^\lambda[\phi]=\lambda \int_0^\infty \int_S \alpha_\infty(\theta)\left[z+\theta {\bf K}\right]\\
&\qquad\qquad\left[\int_{\R^N}\left(\nabla \phi(s-\frac{\theta}{\lambda^2},x)-\nabla \phi(s,x)\right)p^\lambda\left(s,\d x\right)\right]Q(\d\theta,\d z)\d s.
\end{split}
\end{equation*}
Hence, one obtains
\begin{align*}
&\left|I_3^\lambda[\phi]\right|\leq \int_S \sqrt{-\theta}\alpha_\infty(\theta)\left[|z|+|\theta| |{\bf K}|\right]Q(\d\theta,\d z)\\
&\qquad\qquad\qquad\qquad\sup_{h>0} h^{-\frac{1}{2}}\left\|\phi(\cdot+h,\cdot)-\phi\right\|_{L^1(0,\infty;W^{1,\infty}(\R^N))}.
\end{align*}
As consequence, we obtains that there exists some constant, still denoted by $M>0$, such that for any $\lambda>0$ and $\phi\in \mathcal S(\R^+;H^\sigma(\R^N))$
\begin{equation*}
\left|\mathcal T_1^\lambda[\phi]\right|\leq M\left[\lambda^{-1}\|\phi\|_{L^1(0,\infty;H^\sigma(\R^N))}+\|\phi\|_{B_\infty^{\frac{1}{2},1}\left(0,\infty;H^\sigma(\R^N)\right)}\right].
\end{equation*}

\vspace{1ex}

\noindent {\bf \underline{Estimate for $\mathcal T_2^\lambda$}:}\\ 
Recall that $\mathcal T_2^\lambda$ is defined in Lemma \ref{LE24}. Hence it directly follows from this definition that there exists some constant $C>0$ such that, for all $\phi\in \mathcal S(\R^+;H^\sigma(\R^N))$ and $\lambda>0$, one has
\begin{equation*}
|\mathcal T^\lambda_2[\phi]\leq C\|\phi\|_{L^1(0,\infty;W^{2,\infty}(\R^N)}\leq C\|\phi\|_{L^1(0,\infty;H^\sigma(\R^N))}.
\end{equation*}

\vspace{1ex}
\noindent {\bf \underline{Estimate for $\mathcal R^\lambda$}:}\\ 
Recalling Assumption \ref{ASS1} $(i)$, it readily follows that there exists some constant $C>0$ such that for all $\lambda>0$, $\phi\in \mathcal S(\R^+;H^\sigma(\R^N))$ one has
\begin{equation*}
\left|\mathcal R^\lambda[\phi]\right|\leq C\|\phi\|_{L^\infty(\R^+\times \R^N)}\leq C\|\phi\|_{L^\infty(0,\infty;H^\sigma(\R^N))}.
\end{equation*}

\vspace{1ex}
\noindent {\bf \underline{Conclusion}:}\\ 
As a consequence of the above estimates, there exists some constant $C > 0$ such
that for all $\lambda\geq 1$ and $\phi\in\mathcal S\left(\R^+;H^\sigma(\R^N)\right)$ one has:
\begin{align*}
&\left|\left\langle T_{\frac{1}{\lambda^2}}\partial_t\phi,p^\lambda\right\rangle\right|\\
&\leq C\left[|\phi\|_{L^1(0,\infty;H^\sigma(\R^N))}+\|\phi\|_{B_\infty^{\frac{1}{2},1}\left(0,\infty;H^\sigma(\R^N)\right)}+\|\phi\|_{L^\infty(0,\infty;H^\sigma(\R^N))}\right].
\end{align*}
Now since $H^\sigma (\R^N ) = \left(H^{-\sigma}(R^N)\right)'$ is a separable reflexive Banach space, Theorem \ref{THEO-reg} applies and ensures that for each $T > 0$ there exists $M_T > 0$ such
that
\begin{equation*}
\left[p^\lambda\right]_{W^{1-s,p'} (0,T ;H^{-\sigma}(R^N ))}\leq  M_T,\;\forall \lambda\geq 1.
\end{equation*}
Finally the uniform bound \eqref{Linfty} and the above semi-norm estimate completes
the proof of Lemma \ref{LE-sob}.
\end{proof}

Before going to the proof of Theorem \ref{THEO-compact} we will first prove the following regularity lemma
\begin{lemma}\label{LE-reg}
Fix $\sigma>2+\frac{N}{2}$. Let $0<\varepsilon<T$ and $p>2$ be given. Then there exists $M>0$ such that for any $\psi\in \mathcal D(\varepsilon,T;H^\sigma(\R^N))$ and all $\lambda$ large enough the following estimate holds true
\begin{equation*}
\left|\int_0^\infty\int_{\R^N}\partial_t \psi(t,x)p^\lambda(t,\d x)\d t\right|\leq M\|\psi\|_{L^p(\varepsilon,T;H^\sigma(\R^N))}.
\end{equation*}
\end{lemma}

\begin{proof}
Let $\psi\in \mathcal D(\varepsilon, T; H^\sigma(\R^N))$ be given. Let $\phi_\lambda\in \mathcal S(\R^+; H^\sigma (\R^N ))$ be the
function defined by
$$
T_{\frac{1}{\lambda^2}}\phi_\lambda=\psi,\;\forall \lambda>0.
$$
Here $T_{\frac{1}{\lambda^2}}$ is the operator defined in \eqref{Th} with $h = 1/\lambda^2$.
Next note that, due to Lemma \ref{LE-repre}, for all $t\in \geq 0$ one has
\begin{equation}\label{phi-lambda}
\phi_\lambda(t, \cdot) = \psi(t,\cdot) +\int_0^\infty K_{1/\lambda^2}(s)\psi(t-s,\cdot)\d s.
\end{equation}
Hence $\phi_\lambda(t,\cdot)=0$ for all $t\in [0,\varepsilon]$.\\
Next using the same notations as above one obtains due to \eqref{phi-lambda} that the following decomposition holds true:
\begin{equation}\label{eq-44}
\int_{\R^+}\int_{\R^N}\partial_t\psi p^\lambda(t,\d x)\d t=-\mathcal T_1^\lambda[\phi_\lambda]-\mathcal T_2^\lambda[\phi_\lambda]+\mathcal R_1^\lambda[\phi_\lambda].
\end{equation}
Next we shall estimate each of these terms.
We start by the second one by observing that due to the above computations
and Lemma \ref{LE-reg+} $(ii)$, there exists some constant $M > 0$ such that
$$
\left|\mathcal T_2^\lambda[\phi_\lambda]\right|\leq M_2\|\psi\|_{L^1(\varepsilon,T;H^\sigma(\R^N))},\;\forall \lambda>0.
$$
We now estimate the third term arising in \eqref{eq-44}. To that aim we first choose $\lambda>\frac{1}{\varepsilon}$ so that, recalling that $\phi_\lambda(t,\cdot)=0$ for $t\in [0,\lambda^{-1}]\subset [0,\varepsilon]$, $\mathcal R_1^\lambda[\phi_\lambda]$ re-writes
\begin{align*}
\mathcal R_1^\lambda[\phi_\lambda]=&\lambda^2\int_{\varepsilon}^\infty \int_S \left[\alpha(\lambda^2 t,\theta)-\alpha_\infty(\theta)\right]\\
&\qquad\qquad\int_{\R^N}\phi_\lambda(t,x)\left[p^\lambda(t,\d x)-p^\lambda\left(t+\frac{\theta}{\lambda^2},\d x\right)\right] Q(\d\theta,\d z)\d t;
\end{align*}
Recalling \eqref{ASS1} $(i)$, there exists $M>0$ such that for all $\lambda>1/\varepsilon$ one has
\begin{equation*}
\left|\mathcal R_1^\lambda[\phi_\lambda]\right|\leq\lambda^2 Me^{-\beta \varepsilon\lambda^2}\|\phi_\lambda\|_{L^1(\varepsilon,T;H^\sigma(\R^N))}\leq \lambda^2 Me^{-\beta \varepsilon\lambda^2}\|\psi\|_{L^1(\varepsilon,T;H^\sigma(\R^N))}.
\end{equation*}
It remains to estimate the first term in \eqref{eq-44}, namely $\mathcal T^\lambda_1[\phi_\lambda]$. To that aim we will make use of the estimate provided in Lemma \ref{LE-sob} above.
First note that coupling the estimate in this lemma together with those in Lemma \ref{LE29}, for all $T'>0$ and $p>2$ thee exists some constant $M=M(T',p)>0$ such that for any $\lambda\geq 1$ one has
$$
\left[p^\lambda\right]_{B^{\frac{1}{2},p'}_\infty(0,T';H^{-\sigma}(\R^N))}\leq M.
$$
Using this bound and similar computations as the ones given in the above section, for each $p>2$ there exists some constant $M>0$ such that for all $\lambda>1/\varepsilon$ and any $\psi\in \mathcal D(\varepsilon,T;H^\sigma(\R^N))$ one has
$$
\left|\mathcal T_1^\lambda[\phi_\lambda]\right|\leq M\|\psi\|_{L^p(\varepsilon,T;H^\sigma(\R^N))}.
$$
Finally, coupling all the above inequalities together with H\"older inequality completes the proof of the lemma. 

\end{proof}

Using the two above lemmas, namely Lemma \ref{LE-sob} and Lemma \ref{LE-reg} we are in
position to conclude the proof of Theorem \ref{THEO-compact}.\\
\begin{proof}[Proof of Theorem \ref{THEO-compact}]
Note that for any given $R > 0$ and $\sigma > \frac{N}{2} + 2$
then embedding $W_0^{\sigma+1,2}(B R ) \hookrightarrow H^\sigma (R^N )$ is compact. Hence the dual continuous embedding $H^{-\sigma}(R^N ) \hookrightarrow H^{-\sigma-1} (B R )$ is also compact. In addition
we deduce from the above lemmas that for each $0 < \varepsilon < T$ and each $p>2$, there exists $\tilde \lambda>0$ large enough such that the family $\{p^\lambda\}_{\lambda>\tilde\lambda}$ is bounded
 in $L^\infty (\varepsilon, T ; H^{-\sigma}(R^N ))$ while the family ${\partial_t p^\lambda }_{\lambda>\tilde\lambda}$ is bounded in
$L^p (\varepsilon, T ; H^{-\sigma-1}(B_R ))$. Thus Aubin-Lions-Simon lemma ( see \cite{2, 16, 18}) applies
and ensures that Theorem \ref{THEO-compact} holds true.
\end{proof}

\appendix

\section{Various notions of topology on the space of measures}

In this appendix we recall different notions of topology that have been used in this work. We refer for instance to the textbook of Billingsley \cite{3} for more details
and results on this topic.

Here we denote by $\mathcal M$ the set of bounded Borel -- signed -- measures on $\R^N$. We also denote by $\mathcal M^+$ and $\mathcal M^-$ the set of positive and negative Borel measures respectively.

\noindent{\bf Total variation norm on $\mathcal M$:}

The space $\mathcal M$ can be firstly endowed with
the usual total variation norm, denoted by $\|.\|_{TV}$ and defined trough the
Hahn-Jordan decomposition of a signed measure $\mu\in\mathcal M$ as
\begin{equation*}
\mu=\mu^+-\mu^-\text{ with }\mu^\pm\in\mathcal M^+,
\end{equation*}
and $\|\mu\|_{TV}=\mu^+(\R^N)+\mu^-(\R^N)$. Hence $\mathcal M$ endowed with the total variation
norm becomes a Banach space.\\
This strong topology has not been used in this work because translation are usually not continuous for this norm topology.
Indeed if ${\bf q}\in \R^N\setminus\{0\}$ is given,  the map $\theta\in [-1, 0]\mapsto \delta_{{\bf q}\theta}\in \mathcal P$ is not continuous with respect to this norm topology.

\noindent{\bf Narrow topology on $\mathcal M$:}

The space $\M$ can be endowed with the so-called narrow topology which is
defined as the weakest topology on $\M$ such that for each test function $f\in
\mathcal C_b (\R^N)$ the functional
\begin{equation*}
\mu\in \M\mapsto\int_{\R^N} f{\rm d}\mu\in \R, 
\end{equation*}
is continuous. The topological space $\M$ endowed with the narrow topology will
be denoted by $\M_{\rm narr}$.\\
Note that for this topology the map $\theta\mapsto \delta_{{\bf q}\theta}$ is continuous.

Let us also recall the so-called Portemanteau theorem:
\begin{theorem}[Portemanteau Theorem]\label{THEO4} Let $\{\mu_n\}_{n\geq 0}\subset \M$ be a given sequence and
$\mu\in \M$ be given. Then the following properties are equivalent:
\begin{itemize}
\item[(i)] $\mu_n\to \mu$ in $\M_{\rm narr}$;
\item[(ii)] $\int_{\R^N}f{\rm d}\mu_n\to \int_{\R^N}f{\rm d}\mu$ for all Lipschitz continuous function on $\R^N$;
\item[(iii)] $\displaystyle \limsup_{n\to\infty}\int_{\R^N}f{\rm d}\mu_n\leq \int_{\R^N}f{\rm d}\mu$ for every upper semi-continuous function $f$ on $\R^N$ bounded from above;
\item[(iv)] $\displaystyle \liminf_{n\to\infty}\int_{\R^N}f{\rm d}\mu_n\geq \int_{\R^N}f{\rm d}\mu$ for every lower semi-continuous function $f$ on $\R^N$ bounded from below;
\item[(v)] $\displaystyle \lim_{n\to\infty}\mu_n(A)=\mu(A)$ for each $A\in\mathcal B(\R^N)$, the Borel sets of $\R^N$ such that $\mu(\partial A)=0$.
\end{itemize}
\end{theorem}

We can now focus on the subset $\P \subset \M$ of probability measures endowed
with the induced narrow topology.
The topological space $\P_{\rm narr}$ is metrizable with the distance associated to the dual norm of $W^{1,\infty}(\R^N)$. Such a distance is called bounded Lipschitz distance,
denoted by $d_{BL}$ and it reads for each $\mu,\nu\in\P$ as
\begin{equation*}
d_{BL}\left(\mu,\nu\right)=\|\mu-\nu\|_{W^{-1,\infty}(\R^N)}=\sup_{\|f\|_{W^{1,\infty}(\R^N)}\leq 1}\int_{\R^N}f{\rm d}\left(\mu-\nu\right).
\end{equation*}

\noindent{\bf Vague topology or Weak$*$ topology of $\M$:}

The set $\M$ can also be endowed with weaker topology. Denote by $\mathcal C_0(\R^N)$ the
space of continuous functions on $\R^N$ tending to $0$ at infinity. Endowed with
the usual sup norm, it becomes a separable Banach space. Then the Riesz
representation theorem ensures the following dual representation
\begin{equation*}
\M=\left(\mathcal C_0(\R^N)\right)'.
\end{equation*}
Hence $\M$ can be endowed with the weak$∗$ topology, that is also called vague
topology. Note that the close balls are compact due to Banach-Alaoglu theorem.
It is also important to notice that since $\C_0(\R^N)$ is separable, the closed ball are
metrizable and thus sequentially compact. In the sequel the topological space
$\M$ endowed with the vague convergence will be denoted by $\M_{\rm vague}$.

Let us recall the following connexion between narrow and weak$∗$ topology:
\begin{lemma}\label{LE5}
Let $\{\mu_n\}_{n\geq 0}\subset\M^+$ be a given sequence and $\mu\in\M^+$ be given.
Assume that
\begin{equation*}
\lim_{n\to\infty}\mu_n=\mu\text{ in }\M_{\rm vague}\text{ and }\lim_{n\to\infty}\int_{\R^N}{\rm d}\mu_n=\int_{\R^N}{\rm d}\mu,
\end{equation*}
then $\mu_n\to\mu$ narrowly, namely in $\M_{\rm narr}$.
\end{lemma}

\noindent{\bf Young measures:}
The notion of Young measure has also been used in this work. Here we recall some basic facts and properties and we refer the reader to \cite{Castaing, Valadier} and the references cited therein for more details on Young measures theory.

First we say that a function $m \equiv m_t : \R^+ \to \M$ is weakly$∗$ measurable if for each
test function $\varphi\in \C_0 (\R^N )$ the map $t \mapsto\int_{\R^N} \varphi{\rm d}m_t$ is measurable from $\R^+$ to
$\R$. This allows to defined the vector space $L^\infty_{\omega*}(\R^+ ; \M)$ as the set of weakly$∗$
measurable map from $\R^+$ into $\M$ and that is essentially bounded. The elements
of $L^\infty_{\omega*}(\R + ; \P)$ are called Young measures.\\
Now let us recall (see for instance \cite{9}) the following duality representation
\begin{equation*}
L^\infty_{\omega*}(\R^+ ; \M)=\left(L^1(\R^+;\C_0(\R^N))\right)'.
\end{equation*}
Recall also that the Banach space $L^1(\R^+;\C_0(\R^N))$ is separable. Therefore
Banach-Alaoglu theorem implies that bounded set in $L^\infty_{\omega*}(\R^+ ; \M)$ are relatively
(sequentially) compact with respect to the weak$∗$ topology. This leads to the
following fundamental compactness lemma for Young measures.
\begin{lemma}\label{LE6}
Let $\{m^k\}_{k\geq 0}$ be a sequence of Young measures, that is elements of $L_{\omega*}^\infty(\R^+;\P)$. Then there exists a subsequence $\{k_n\}_{n\geq 0}$ and a map of measures $m^\infty\equiv m^\infty_t\in L_{\omega*}^\infty(\R^+;\M^+)$ such that
\begin{equation*}
\int_{\R^N} {\rm d}m_t^\infty\leq 1\;a.e.\;t\in\R^+,
\end{equation*}
and such that for all $f\in L^1\left(\R^+;\C_0(\R^N)\right)$ one has
\begin{equation}\label{12}
\lim_{n\to\infty}\int_{\R^+\times\R^N}f(t,x)m_t^{k_n}\left({\rm d}x\right)dt=\int_{\R^+\times\R^N}f(t,x)m_t^{\infty}\left({\rm d}x\right)dt.
\end{equation}
\end{lemma}
Similarly to Lemma \ref{LE5}, one has the following result
\begin{lemma}[Tightness Lemma]\label{LE7}
With the same notations as in Lemma \ref{LE6}, if
the limit measure satisfies the tightness condition
\begin{equation*}
\int_{\R^N} {\rm d}m_t^\infty= 1\;a.e.\;t\in\R^+,
\end{equation*}
then the above convergence, namely \eqref{12} holds for the so-called narrow topology
of Young measures, that reads as for all test function $f\in L^1\left(\R^+;\C_b(\R^N)\right)$
one has
\begin{equation*}
\lim_{n\to\infty}\int_{\R^+\times\R^N}f(t,x)m_t^{k_n}\left({\rm d}x\right)dt=\int_{\R^+\times\R^N}f(t,x)m_t^{\infty}\left({\rm d}x\right)dt.
\end{equation*}
\end{lemma}

\section{A linear delay differential equation}

In this appendix we come back to the linear delay differential \eqref{DDeq} and we will prove that 
$$
\alpha(t,\theta)\to a_\infty(\theta)\text{ as }t\to\infty,
$$
uniformly for $\theta\in [-1,0]$ and with an exponential rate, that re-writes as $\frac{y(t+\theta)}{y(t)}\to e^{\gamma \theta}$ as $t\to\infty$ uniformly with respect to $\theta\in [-1,0]$ and exponentially fast.

To that aim, we consider the equation
\begin{equation}\label{DDE}
\begin{split}
&y'(t)=\int_{[-1,0]}y(t+\theta)K(\d\theta),\;t>0,\\
&y(\theta)=y^0(\theta),\;\forall \theta\in [-1,0].
\end{split}
\end{equation}
wherein $K\in \M^+([-1,0])$ denotes a bounded positive Borel measure on $[-1,0]$ with $K\left([-1,0]\right)>0$ and $y^0\in \C\left([-1,0];\R^+\right)$ with $y^0(0)>0$ so that $y(t)>0$ for all $t>1$.\\
Next to investigate the large time behaviour of the above equation let us introduce the
history function $u \equiv u(t, \theta)$ defined by
\begin{equation*}
u(t,\theta)=y(t+\theta),\;t\geq 0,\;\theta\in [-1,0].
\end{equation*}
This function formally satisfies the following Cauchy problem:
\begin{equation}\label{17}
\begin{cases}
\partial_t u-\partial_\theta u=0,\;t>0,\;\theta\in [-1,0],\\
\partial_\theta u(t,0)=\int_{[-1,0]} u(t,\theta)K({\d}\theta),\\
u(0,.)=y^0\in \C\left([-1,0]\right).
\end{cases}
\end{equation}
Such a linear functional differential problem has been extensively studied in the
literature. We refer for instance to \cite{17, 21}, the monographs \cite{11, 14} and the
references cited therein.\\
Here to be more precise we consider the Banach space $Y=\R\times \C\left([-1,0]\right)$ and $Y_0=\{0\}\times \C\left([-1,0]\right)$ as well as the linear operator $A:D(A)\subset Y\to Y$ defined by
\begin{equation*}
D(A)=\{0\}\times \C^1([-1,0]),\;\;A\begin{pmatrix} 0\\\varphi\end{pmatrix}=\begin{pmatrix} -\varphi'(0)+\int_{[-1,0]} \varphi(\theta)K({\rm d}\theta)\\ \varphi'\end{pmatrix}.
\end{equation*}
Then setting $U(t)=\begin{pmatrix}
0\\ u(t,.)\end{pmatrix}$, Problem \eqref{17} re-writes as the following abstract Cauchy problem
\begin{equation*}
\frac{dU(t)}{dt}=AU(t),\;t>0\text{ and }U(0)=\begin{pmatrix}
0\\ y^0\end{pmatrix}\in Y_0=\overline{D(A)}.
\end{equation*}
Following \cite{17} (see also the references therein) this problem generates a strongly
continuous linear semigroup $\{T(t)\}_{t\geq 0}$ on $Y_0$ with infinitesimal generator $A_0$, the part of $A$ in $Y_0$, defined as
\begin{equation*}
D(A_0)=\{x\in D(A):\;Ax\in Y_0\}\text{ and }A_0x=Ax\;\forall x\in D(A_0).
\end{equation*}
In addition, the essential growth rate $\omega_{0,ess}(A_0)$ satisfies $\omega_{0,ess}(A_0)=-\infty$ so that, due to usual results for spectral theory (see for instance the monograph \cite{22}), the spectrum of $A_0$ only consists in point spectrum and one has
\begin{equation*}
\sigma\left(A_0\right)=\{z\in \mathbb C:\;\Delta(z)=0\},
\end{equation*}
wherein the function $\Delta:\mathbb C\to \mathbb C$ is defined by
\begin{equation}\label{18}
\Delta(z)=z-\int_{[-1,0]} e^{\theta z}K({\rm d}\theta).
\end{equation}
In addition, the growth rate $\omega_0 (A_0)$ of the linear semigroup $\{T(t)\}_{t\geq 0}$ is obtained by
\begin{equation*}
\omega_0 (A_0)=\max\{\Re z:\;z\in \sigma(A_0)\}.
\end{equation*}
Let us now consider the unique $\gamma > 0$ solution of the equation $\Delta(z) = 0$. Next the following lemma holds true:
\begin{lemma}\label{LE14}
Let $z\in\mathbb C$ be given such that $\Delta(z)=0$. Then the following properties hold true:
\begin{itemize}
\item[(i)] One has $\Re(z)\leq \gamma$ and $\Re(z)=\gamma$ $\Longrightarrow$ $z=\gamma$.
\item[(ii)] If $\Re(z)\geq 0$ then $|\Im(z)|\leq K\left([-1,0]\right)$.
\end{itemize}
\end{lemma}
The proof of $(i)$ follows the same lines as the one of Lemma \ref{LE-carac} while the proof of $(ii)$ is straightforward.

As a direct corollary of the above lemma, one obtains that
\begin{equation*}
\omega_0(A_0)=\gamma,
\end{equation*}
and there exists $\varepsilon>0$ such that for all $z\in \mathbb C$:
\begin{equation}\label{19}
\Delta(z)=0\text{ and }\Re(z)\geq \gamma-\varepsilon\;\Longrightarrow\;z=\gamma.
\end{equation}
Moreover let us notice that
\begin{equation*}
\Delta'(\gamma)=1+\int_{[-1,0]} (-\theta)e^{\theta \gamma}K({\rm d}\theta)\in [1,\infty).
\end{equation*}
This means that the dominant eigenvalue $\gamma$ is simple.\\
And therefore, the usual spectral theory ensures that there
exists $\alpha_1 > 0$ and $\beta > 0$ such that
\begin{equation*}
y(t)=e^{\gamma t}\left(\alpha_1+O(e^{-\beta t}\right)\text{ as }t\to\infty.
\end{equation*}
Hence we obtain
$$
\frac{y(t+\theta)}{y(t)}\to e^{\gamma\theta}\text{ as }t\to \infty,
$$
uniformly with respect to $\theta\in [-1,0]$ and exponentially fast.

\begin{thebibliography}{99}

\bibitem{1} H. Amann, Compact embeddings for vector-valued Sobolev and Besov
spaces, Glas. Mat. 35 (2000), 161--177.

\bibitem{2} J.-P. Aubin, Un th\'eor\`eme de compacit\'e, C. R. Acad. Sci. Paris 256 (1963),
5042--5044.

\bibitem{Bates}   P.W. Bates, On some nonlocal evolution equations arising in materials science, In: {Nonlinear dynamics and evolution equations} (Ed. by H. Brunner, X. Zhao and X. Zou), pp. 13-52, Fields Inst. Commun., 48, AMS, Providence,  2006.


\bibitem{3} P. Billingsley, Convergence of Probability Measures, Wiley, 2nd ed., 1999.


\bibitem{Castaing} C. Castaing, P. Raynaud de Fitte and M. Valadier,
Young measures on topological spaces: With applications in control theory and probability theory. Mathematics and its Applications, 571. Kluwer Academic Publishers, Dordrecht, 2004. xii+320 pp.



\bibitem{5} E. Chasseigne, M. Chaves and J.D. Rossi, Asymptotic behavior for nonlocal
diffusion equations, J. Math. Pures Appl. 86 (2006), 271--291.

\bibitem{6} C. Cosner, J. D\'avila and S. Mart\'inez, Evolutionary stability of ideal free
nonlocal dispersal, J. Bio. Dyn., 6 (2012), 395--405.

\bibitem{7} J. Coville, Remarks on the strong maximum principle for nonlocal operator,
Elec. J. Diff. Eqs., 68 (2008), 1--10.

\bibitem{8} J. Coville, L. Dupaigne, On a non-local reaction diffusion equation arising
in population dynamics, Proc. Math. Roy. Soc. of Edinburgh Sect. A 137
(2007), 1--29.

\bibitem{9} P. Cembranos and J. Mendoza, Banach Spaces of Vector-Valued Functions,
vol. 1676, Springer-Verlag, Berlin, 1997.


\bibitem{11} O. Diekmann, S.A. van Gils, S.M. Verduyn Lunel, H.-O. Walther, Delay
Equations, Function-, Complex-, and Nonlinear Analysis, Springer-Verlag,
New York, 1995.



\bibitem{Fife1} P. Fife, Some nonclassical trends in parabolic and parabolic-like evolutions. Trends in nonlinear anal-
ysis, 153--191, Springer, Berlin, 2003.

\bibitem{Fife2} P. Fife and X. Wang, A convolution model for interfacial motion: the generation and propagation of
internal layers in higher space dimensions, Adv. Differential Equations, 3 (1998), 85--110.


\bibitem{14} J.K. Hale, S.M. Verduyn Lunel, Introduction to Functional Differential
Equations, Springer-Verlag, New York, 1993.

\bibitem{Ignat-Rossi} L.I. Ignat, J.D. Rossi, A nonlocal convection-diffusion equation, {J. Funct. Anal.,} 251 (2007), 399--437.


\bibitem{Jin-Zhao} Y. Jin, X.-Q.  Zhao,  {Spatial dynamics of a periodic population model with dispersal}, {{Nonlinearity,}} 22 (2009),  1167--1189.

\bibitem{15} T. Laurent, B. Rider, M. Reed, Parabolic behaviour of a hyperbolic delay
equation, SIAM Journal Math. Anal. 38 (2006), 1--15.

\bibitem{Li-2018} W.T. Li, J.B. Wang, X.-Q. Zhao, {Spatial dynamics of a nonlocal dispersal population model in a shifting environment},
{J. Nonlinear Sci.} https://doi.org/10.1007/s00332-018-9445-2


\bibitem{16} J.L. Lions, Quelques m\'ethodes de r\'esolution des probl\`emes aux limites non
lin\'eaires, 1969, Paris: Dunod-Gauth. Vill.

\bibitem{17} Z. Liu, P. Magal and S. Ruan, Projectors on the generalized eigenspaces for
functional differential equations using integrated semigroups, J. Diff. Eqs,
244 (2008), 1784--1809.

\bibitem{Lutscher} F. Lutscher, E. Pachepsky, and M. A. Lewis, {The effect of dispersal patterns
on stream populations}, SIAM J. Appl. Math., 65 (2005), 1305--1327.

\bibitem{Meerschaert} M.M. Meerschaert and P. Straka. {Semi-Markov approach to continuous time random walk limit processes}, The Annals of Probability, 42(4) (2014), 1699--1723.

\bibitem{Majerek} D. Majerek, W. Nowak and W. Zieba. {Conditional strong law of large number}. Int. J. Pure Appl. Math, 20(2) (2005), 143--156.

\bibitem{18} J. Simon, Compact sets in the space $L^p (O, T ; B)$, Annali di Matematica
Pura ed Applicata 146 (1986), 65--96.

\bibitem{19} J. Simon, Sobolev, Besov and Nikolskii fractional spaces: Imbeddings and
comparisons for vector valued spaces on an interval, Ann. Math. Pura Appl.
157 (1990), 117--148.

\bibitem{SV} D. W. Stroock and S. S. Varadhan, S. S. (2007). Multidimensional diffusion processes. Springer.

\bibitem{20} M. Suwa and K. Yoshino, A characterization of tempered distributions with
support in a cone by the heat kernel method and its applications, J. Math.
Sci. Univ. Tokyo, 11 (2004), 75--90.

\bibitem{21} H.R. Thieme, Semiflows generated by Lipschitz perturbations of non-
densely defined operators, Differential Integral Equations 3 (1990), 1035--1066.
\bibitem{22} G.F. Webb, Theory of Nonlinear Age-Dependent Population Dynamics,
Marcel Dekker, New York, 1985.

\bibitem{Valadier} M. Valadier, Young measures. In Methods of Noncovex Analysis, A. Cellina Ed., Lecture Notes in Math. 1446, Springer Verlag, Berlin 1990, 152--188.
\end{thebibliography}
\end{document}